\documentclass{amsart}[12pt]

\usepackage{amsmath, amsthm, amscd, amsfonts,}
\usepackage{amssymb}
\usepackage{todonotes}

\DeclareMathOperator{\otp}{otp}

\DeclareMathOperator{\dom}{dom}
\DeclareMathOperator{\cf}{cf}
\DeclareMathOperator{\Col}{Col}

\DeclareMathOperator{\Lev}{Lev}
\DeclareMathOperator{\supp}{supp}

\DeclareMathOperator{\ZFC}{ZFC}
\DeclareMathOperator{\TSP}{SATP}
\DeclareMathOperator{\GCH}{GCH}

\def\MPB{{\mathbb{P}}}
\def\MQB{{\mathbb{Q}}}

\def\a{\alpha}

\newtheorem{theorem}{Theorem}[section]
\newtheorem{lemma}[theorem]{Lemma}

\newtheorem{corollary}[theorem]{Corollary}
\newtheorem{notation}[theorem]{Notation}
\newtheorem{definition}[theorem]{Definition}

\newtheorem{remark}[theorem]{Remark}
\newtheorem{claim}[theorem]{Claim}
\newtheorem{question}[theorem]{Question}
\numberwithin{equation}{section}

\usepackage{hyperref}
\begin{document}
\title[The Special Aronszajn Tree Property]{The Special Aronszajn Tree Property }
\author[M. Golshani]{Mohammad Golshani}
\address{
School of Mathematics, Institute for Research in Fundamental Sciences (IPM), P.O. Box:
19395-5746, Tehran-Iran.
}
\email[M. Golshani]{golshani.m@gmail.com}

\author[Y. Hayut]{Yair Hayut}
\address{
Universit\"{a}t Wien \\
Kurt G\"{o}del Research Center for Mathematical Logic \\
W\"{a}hringer Stra{\ss}e 25 \\
1090 Wien \\
Austria
}
\email[Y. Hayut]{yair.hayut@mail.huji.ac.il}
\date{}

\thanks{The first author's research has been supported by a grant from IPM (No. 971030417). The second author's research has been supported by the FWF Lise Meitner Grant, 2650-N35}
\begin{abstract}
Assuming the existence of a proper class of supercompact cardinals, we force a generic extension in which, for every regular cardinal $\kappa$, there are $\kappa^+$-Aronszajn trees, and all such trees are special.
\end{abstract}

\thanks{ } \maketitle

\section{Introduction}
Aronszajn trees are of fundamental importance in combinatorial set theory, and two of the most interesting problems
about them, are the problem of their existence (the Tree Property), and the problem of their specialization (the Special Aronszajn Tree Property).

Given a regular cardinal $\kappa$, a \emph{$\kappa$-Aronszajn tree}
is a tree of height $\kappa$, where all of its levels have size less than $\kappa$ and it has no cofinal branches of size $\kappa$. The \emph{Tree Property} at $\kappa$ is the assertion ``there are no $\kappa$-Aronszajn trees''.

By a theorem of K\"{o}nig, the tree property holds at $\aleph_0$, while  by a result of Aronszajn, the tree property fails at $\aleph_1$.
The problem of the tree property at higher cardinals is more complicated and is independent of $\ZFC$. An interesting and famous
question of Magidor asks if the tree property can hold at all regular cardinals bigger than $\aleph_1$, and though the problem
is widely open, there are many works towards a positive answer (a partial list includes \cite{Mitchell1972}. \cite{Abraham1983}, \cite{MagidorShelah1996}, \cite{Neeman2014} and more).

In this paper, we are interested in the problem of specializing Aronszajn trees at successors of regular cardinals.

\begin{definition}
A $\lambda^{+}$-Aronszajn tree $T$, on a successor cardinal $\lambda^{+}$, is \emph{special}, if there exists a function $f\colon T \to \lambda$ such that for every $x, y$ in $T$, if $x <_T y$, then $f(x) \neq f(y)$.
\end{definition}
The specialization function, $f$, witnesses the fact that $T$ has no cofinal branches (as the restriction of $f$ to a cofinal branch is an injective function from a set of size $\lambda^+$ to $\lambda$). Thus, if $T$ is special, then it remains Aronszajn in any larger model of $\ZFC$ in which $\lambda^{+}$ is a cardinal.

For an uncountable regular cardinal $\kappa$, let $\TSP(\kappa)$, the \emph{Special Aronszajn Tree Property} at $\kappa$, be the assertion ``there are
$\kappa$-Aronszajn trees and all such trees are special''.
By Baumgartner-Malitz-Reinhardt \cite{baumgartner}, $\text{MA}+\neg \text{CH}$ implies  $\TSP(\aleph_1)$. Laver-Shelah \cite{laver-shelah} extended this result to get
$\TSP(\kappa^+)$, for $\kappa$ regular, starting from a weakly compact cardinal bigger than $\kappa$. Large cardinals seem to be unavoidable when dealing with specialization of trees of uncountable height, see \cite{Rinot2017}.

In this paper, we force the Special Aronszajn Tree Property at many successors of regular cardinals.
First, we consider the case of forcing the Special Aronszajn Tree Property
 at both $\aleph_1$ and $\aleph_2$, and prove the following theorem.
\begin{theorem}
\label{main theorem2} Assume there exists a weakly compact cardinal. Then there is a generic extension of the universe
in which the Special Aronszajn Tree Property holds at both $\aleph_1$ and $\aleph_2$.
\end{theorem}
Then we  consider the problem of specializing Aronszajn trees at infinitely many successive cardinals,
and prove the following theorem.
\begin{theorem}
\label{main theorem3}
Assume there are infinitely many supercompact cardinals. Then there is a forcing extension of the universe in which the Special Aronszajn Tree Property
holds at $\aleph_{n}$ for $0 < n < \omega$.
\end{theorem}
The above result can be extended to get the Special Aronszajn Tree Property at all $\aleph_{\alpha+n}$'s, where $\alpha$
is any limit ordinal and $1<n < \omega$. Finally, we use a class-sized iterated forcing construction to get the following result.
\begin{theorem}
\label{main theorem}
Assume  there is a proper  class of supercompact cardinals with no inaccessible limit.
  Then there is a $\ZFC$-preserving class forcing extension of the universe,  in which the Special Aronszajn Tree Property
holds at the successor of every regular cardinal.
\end{theorem}
Our forcing notions are design to specialize trees at a double successor cardinal, in a way that allow us to specialize trees at many cardinals simultaneously. Using Baumgartner's forcing, we can also specialize all $\aleph_1$-trees. The possibility of specialization of Aronszajn trees at the successor of a singular cardinal or the successor of an inaccessible cardinal remains open.

It is clear that if $T$ is a special $\kappa$-Aronszajn tree, then $T$ is not $\kappa$-\emph{Suslin}; so the problem of making all $\kappa$-Aronszajn trees special is tightly connected to the $\kappa$-Suslin hypothesis, which asserts that there are no $\kappa$-Suslin trees.
Let the \emph{Generalized Suslin Hypothesis} be the assertion ``the $\kappa$-Suslin hypothesis holds at all uncountable regular cardinals $\kappa$''.
The consistency of the Generalized Suslin Hypothesis is an old and major open question in set theory. As a corollary of Theorem \ref{main theorem}, we obtain the following partial answer to it.
\begin{corollary}
Assume there are class many supercompact cardinals with no inaccessible limit.
  Then there is a $\ZFC$-preserving class forcing extension of the universe, in which the Generalized Suslin Hypothesis holds at the successor of every regular cardinal.
\end{corollary}

The paper is organized as follows. In Section \ref{section2} we prove Theorem \ref{main theorem2}. To do this, we first introduce Baumgartner's forcing
 for specializing $\aleph_1$-Aronszajn trees, and discuss some of its basic properties. Then we introduce a new forcing notion, which specializes \emph{names} for $\aleph_2$-Aronszajn trees, and show that it has many properties in common with the Laver-Shelah forcing for specializing  $\aleph_2$-Aronszajn trees. Finally we show how the above results can be combined to define a forcing iteration which gives the proof of Theorem \ref{main theorem2}. This part contains almost all technical difficulties which appear in the general case.
 
In Section \ref{section:abstract}, we restate the main technical lemmas of Section \ref{section2} in a general way which is suitable for the purposes of Section \ref{section3} and Section \ref{section4}.
 In Section \ref{section3} we prove Theorem \ref{main theorem3}, and finally, in Section \ref{section4}, we show how to iterate the forcing notion of section \ref{section3} to prove Theorem \ref{main theorem}.

Our notations are mostly standard. For facts about forcing and large cardinals we refer the reader to \cite{jech2003}.

We force downwards and we always assume that our forcing notions are separative, namely for pair of conditions $p, q$ in a forcing notion $\mathbb{P}$, $p \leq q$ means that $p$ is stronger than $q$ and equivalently $p \Vdash q\in \dot{G}$ (where $\dot{G} = \{ \langle p, \check{p}\rangle \mid p \in \mathbb{P}\}$ is the canonical name for the generic filter). Also if $\MPB$ is a forcing notion in the ground model $V$, when writing $V[G_{\MPB}]$, we assume $G_{\MPB}$ is a $\MPB$-generic filter over $V$.
\section{The Special Aronszajn Tree Property at \texorpdfstring{$\aleph_1$}{aleph1} and \texorpdfstring{$\aleph_2$}{aleph2}}
\label{section2}
In this section we prove Theorem \ref{main theorem2}. In Subsection \ref{Baumgartner forcing for specializing trees}, we review Baumgartner's  forcing for specializing
$\aleph_1$-Aronszajn trees. In Subsection \ref{Specializing names for trees}, we introduce a  forcing notion for specializing names of $\aleph_2$-Aronszajn trees.
The forcing is a variant of the Laver-Shelah forcing \cite{laver-shelah}, where instead of specializing $\aleph_2$-Aronszajn trees, we specialize names of $\aleph_2$-Aronszajn trees.
In Subsection \ref{Definition of the main forcing}, we define the main forcing iteration, and in Subsection \ref{Properties of the forcing notion P}, we prove its basic properties. The main technical part
is to show that the forcing iteration satisfies the $\kappa$-chain condition, where $\kappa$ is the weakly compact cardinal we start with.
Finally in Subsection \ref{Completing the proof of Theorem main theorem2} we complete
the proof of Theorem \ref{main theorem2}.
\subsection{Baumgartner's  forcing for specializing \texorpdfstring{$\aleph_1$}{aleph1}-Aronszajn trees}
\label{Baumgartner forcing for specializing trees}
In this subsection we briefly review Baumgartner's forcing for specializing $\aleph_1$-Aronszajn trees,
and refer to \cite{baumgartner1983} for more details on the results of this subsection.
\begin{definition}
Let $T$ be an $\aleph_1$-Aronszajn tree. The conditions in Baumgartner's  forcing for specializing $T$, $\mathbb{B}(T)$,
are partial functions $f: T \to \omega$ such that
\begin{enumerate}
\item $\dom(f) \subseteq T$ is finite.
\item If $s, t \in \dom(f)$ and $s <_T t$, then $f(s) \neq f(t)$.
\end{enumerate}
The order on $\mathbb{B}(T)$ is the reverse inclusion.
\end{definition}
Let us state the basic properties of the forcing notion $\mathbb{B}(T)$. The proof of the following lemma can be found in \cite[page 274]{jech2003}
\begin{lemma}
\label{basic properties of B(T)}
\begin{itemize}
\item [(a)] $\mathbb{B}(T)$ is c.c.c.

\item [(b)] In the generic extension by $\mathbb{B}(T)$, the tree $T$ is specialized; in fact if $G$ is $\mathbb{B}(T)$-generic over the ground model $V$,
then $F= \bigcup G$ is a specializing function from $T$ to $\omega$.
\end{itemize}
\end{lemma}

\begin{definition}
\label{def:B(V)}
Baumgartner's forcing for specializing all $\aleph_1$-Aronszajn trees, $\mathbb{P}$, is defined as the finite support iteration
\[
\mathbb{P} = \langle   \langle  \MPB_\alpha \mid \alpha \leq  2^{\aleph_1}  \rangle,  \langle \dot{\MQB}_\alpha \mid \alpha < 2^{\aleph_1}   \rangle \rangle
\]
of forcing notions where
\begin{enumerate}
\item For each $\alpha < 2^{\aleph_1}, \Vdash_{\MPB_\alpha}$``$\dot{\MQB}_\alpha =\mathbb{B}(\dot{T}_\alpha)$'', for some $\MPB_\alpha$-name
$\dot{T}_\alpha$ which is forced by $1_{\MPB_\alpha}$ to be an $\aleph_1$-Aronszajn tree.

\item If $\dot{T}$ is a $\mathbb{P}$-name for an $\aleph_1$-Aronszajn tree, then for some $\alpha < 2^{\aleph_1}, \dot{T}$ is a $\MPB_\alpha$-name
and $\Vdash_{\MPB_\alpha}$``$\dot{T}=\dot{T}_\alpha$''.
\end{enumerate}
\end{definition}

Let us mention some basic properties of $\mathbb{P}$.
\begin{lemma} \label{basic properties of Baumgartner P}
\begin{itemize}
\item [(a)] $\mathbb{P}$ is c.c.c.

\item [(b)] In the generic extension by $\mathbb{P}, 2^{\aleph_0}=(2^{\aleph_1})^V$ and all $\aleph_1$-Aronszajn trees are specialized.
\end{itemize}
\end{lemma}
\begin{proof}
$($a$)$ Follows from Lemma \ref{basic properties of B(T)}$($a$)$ and the Solovay-Tennenbaum theorem that the finite support iteration of c.c.c.\ forcing notions is c.c.c., \cite{SolovayTennenbaum}.

$($b$)$ Follows from  Lemma \ref{basic properties of B(T)}$($b$)$  and Definition \ref{def:B(V)}(2).
\end{proof}

In the above definition of $\mathbb{P}$, we used some underlying bookkeeping method which was used in order to pick the names $\dot{T}_\alpha$. We will need a minor generalization of this. Let $\mathcal{T}$ be a function such that for every c.c.c.\ forcing notion $\mathbb{R}$, $\mathcal{T}(\mathbb{R})$ is an $\mathbb{R}$-name for an $\aleph_1$-Aronszajn tree. We do not require that every name for an $\aleph_1$-Aronszajn tree is enumerated by $\mathcal{T}$. Let
$$\mathbb{P}_\gamma(\mathcal{T})=\langle   \langle  \mathbb{P}_\alpha(\mathcal{T}) \mid \alpha \leq  \gamma  \rangle,  \langle \dot{\MQB}_\alpha(\mathcal{T}) \mid \alpha <\gamma   \rangle \rangle$$
be the finite support iteration of forcing notions of length $\gamma$, where for each $\alpha < \gamma$, $\Vdash_{\mathbb{P}_\a(\mathcal{T})}$``$\dot{\MQB}_\alpha(\mathcal{T}) =\mathbb{B}(\mathcal{T}(\mathbb{P}_\alpha(\mathcal{T}))$''.

Note that for every $\mathcal{T}$ as above and every ordinal $\gamma$, $\mathbb{P}_\gamma(\mathcal{T})$ is c.c.c.,\ as a finite support iteration of c.c.c.\ forcing notions.

The following lemma will be used in the course of proving Theorem \ref{main theorem2}.
\begin{lemma}
\label{lemma: not adding new branches by baumgartner forcing}
Let $\mathcal{T}$ be as above. Let $S$ be a tree of height $\omega_1$ and arbitrary width and let $\gamma$ be an ordinal. Then $\mathbb{P}_\gamma(\mathcal{T})$ does not introduce new branches to $S$.
\end{lemma}
\begin{proof}
Let us show that $\mathbb{P}_\gamma(\mathcal{T}) \times \mathbb{P}_\gamma(\mathcal{T})$ is c.c.c. Let $\mathcal{T}'$ be the following function:
\begin{itemize}
\item If $\alpha < \gamma$, then $\mathcal{T}'(\mathbb{P}_\alpha(\mathcal{T}'))=\mathcal{T}(\mathbb{P}_\alpha(\mathcal{T}))$. In particular,
$\mathbb{P}_\alpha(\mathcal{T}') \cong \mathbb{P}_\alpha(\mathcal{T})$, for all $\alpha \leq \gamma$.
\item If $\gamma \leq \alpha < \gamma+\gamma$, and if $\beta < \gamma$
is such that $\alpha=\gamma+\beta$, then $\mathcal{T}'(\mathbb{P}_\alpha(\mathcal{T}'))=\mathcal{T}(\mathbb{P}_\beta(\mathcal{T}))$.
\end{itemize}
Note that if $\alpha=\gamma+\beta$, where $\beta < \gamma$, then $\Vdash_{\mathbb{P}_\alpha(\mathcal{T}')}$``$\mathcal{T}'(\mathbb{P}_\alpha(\mathcal{T}'))$ is a special Aronszajn tree'', and in particular it is Aronszajn.
It then follows that  the forcing iteration $\mathbb{P}_{\gamma+\gamma}(\mathcal{T}')$ is c.c.c., and by the definition of $\mathcal{T}'$, one can easily verify that $\mathbb{P}_\gamma(\mathcal{T}) \times \mathbb{P}_\gamma(\mathcal{T}) \cong \mathbb{P}_\gamma(\mathcal{T}) \ast \dot{\mathbb{P}}_\gamma(\mathcal{T}) \cong \mathbb{P}_{\gamma+\gamma}(\mathcal{T}')$. The lemma follows from \cite[Lemma 1.3]{Unger2015}.
\end{proof}

The following definition appears in the literature under various names and notations. For an example in which the following concept is used extensively, see \cite{ShelahThomas1997}. 
\begin{definition}
Let $\langle \MPB_\alpha, \dot{\MQB}_\beta \mid \beta < \delta,\,\alpha \leq \delta\rangle$ be a $<\mu$-support iteration of forcing notions,  and let $I \subseteq \delta$. We define $\MPB_I$, by induction on $\otp(I)$, to be the $<\mu$-support iteration $\MPB_I=\langle \MPB_{I \cap \alpha},\, \dot{\MQB}_{I \cap \beta} \mid \beta \in I,\, \alpha \in I \cup \{\sup(I) + 1\}\rangle$ of forcing notions, such that:
\begin{enumerate}
\item If  $\dot{\MQB}_\beta$ is forced by the weakest condition of $\MPB_{\beta}$ to be equivalent to a specific $\MPB_{I \cap \beta}$-name, then $\dot{\MQB}_{I \cap \beta}$ is such a $\MPB_{I \cap \beta}$-name.
\item Otherwise $\Vdash_{\MPB_{I \cap \beta}}$``$\dot{\MQB}_{I \cap \beta}$ is the trivial forcing''.
\end{enumerate}
We say that $\MPB_I$ is a sub-iteration of $\MPB$ if the second case does not occur.
\end{definition}

Note that $\MPB_I$ is always a regular subforcing of $\MPB$.
\begin{lemma}\label{lemma: not adding new branches by baumgartner forcing mod subiteration}
Let $\mathbb{P}_\delta(\mathcal{T})$ be an iteration of Baumgartner's forcing as above, and let $I \subseteq \delta$ be a set of indices such that $\MPB_I$ is a subiteration of $\MPB_{\delta}$. Let $S$ be a tree of height $\omega_1$ in the generic extension by $\MPB_I$. Then the quotient forcing $\MPB_{\delta} / \MPB_I$ does not add a new branch to $S$.
\end{lemma}
\begin{proof}
We would like to claim that the quotient forcing $\MPB_{\delta} / \MPB_I$ is equivalent to a finite support iteration of Baumgartner's forcing in the generic extension by $\MPB_I$. 

First, let us observe that generally, if $\mathbb{P} \cong \mathbb{Q} \ast \dot{\mathbb{R}}$ has the property that $\mathbb{P} \times \mathbb{P}$ is c.c.c. then so do $\mathbb{Q}$ and $\dot{\mathbb{R}}$ (in the generic extension by $\mathbb{Q}$). Indeed, if $\{\langle q_i, q'_i\rangle \mid i < i_\star\}$ is an antichain in $\mathbb{Q} \times \mathbb{Q}$ then $\{\langle (q_i, 1), (q'_i, 1)\rangle \mid i < i_\star\}$ is an antichain in $\mathbb{P}\times \mathbb{P}$. Similarly, if $\{\langle \dot{r}_i, \dot{r}'_i\rangle \mid i < i_\star\}$ is forced by a condition $q\in \mathbb{Q}$ to be an antichain in $\dot{\mathbb{R}} \times \dot{\mathbb{R}}$ then $\{\langle (q, \dot{r}_i), (q, \dot{r}'_i) \rangle \mid i < i_\star\}$ is an antichain in $\mathbb{P} \times \mathbb{P}$. 

Let us apply those observations for $\MPB_I$. Let $J = \delta \setminus I$ and let us define inductively a function $\mathcal{T}'$ such that in the generic extension by $\mathbb{P}_I$, $\MPB_\delta / \MPB_I \cong \MPB_{\otp J}(\mathcal{T}')$. 

By induction on $\beta \in J$, let us define $\mathcal{T}'(\MPB_{\otp (J \cap \beta)}(\mathcal{T}'))$ to be the $\omega_1$-tree $\mathcal{T}(\MPB_\beta)$, partially evaluated by the generic for $\MPB_{I\cap \beta}$. Let us claim that it is Aronszajn in $\MPB_I$. Indeed, it is an Aronszajn tree in the extension by $\MPB_{I \cap \beta} \ast \MPB_{\otp(J\cap \beta)}(\mathcal{T}')$, which is equivalent to $\MPB_\beta$ (using an inductive assumption). Thus, we have to verify that it remains Aronszajn in the extension by $\MPB_{I \setminus \beta}$. Clearly, this is a sub-iteration of the forcing $\MPB_{\delta} / \MPB_\beta$. Thus, it is a regular subforcing and in particular, since  
$\MPB_{\delta} / \MPB_\beta \times \MPB_{\delta} / \MPB_\beta$ is c.c.c. (as a quotient of productively c.c.c.\ forcing), so is $\MPB_{I \setminus \beta}$. We conclude that it cannot add new branches to a tree of height $\omega_1$, as needed. 
\end{proof}
\subsection{Specializing names for \texorpdfstring{$\aleph_2$}{aleph2}-Aronszajn trees}
\label{Specializing names for trees}
In this subsection, we define a forcing notion for specializing names of $\aleph_2$-Aronszajn trees.
\begin{definition}
Let $V$ be the ground model, $\kappa$ be an inaccessible cardinal in $V$ and suppose that $\MPB \ast \dot{\MQB}$ is a two step iterated forcing which is $\kappa$-c.c.\ and makes $\kappa=\aleph_2$. Let
$\dot{T}$ be a $\MPB \ast \dot{\MQB}$-name for a $\kappa$-Aronszajn tree. We may assume that $\dot{T}$
is forced to be a tree on $\kappa \times \omega_1$ and that the $\alpha$-th level of it is forced to be $\{\alpha\} \times \omega_1$. Let $\mathbb{B}_\MQB(\dot{T})$ be the following forcing notion as it is defined in $V[G_{\MPB}]$:

Conditions in $\mathbb{B}_\MQB(\dot{T})$ are partial functions $f: \kappa \times \omega_1 \to \omega_1$
such that:
\begin{enumerate}
\item $\dom(f) \subseteq \kappa \times \omega_1$ is countable.
\item If $s, t \in \dom(f)$ and $f(s) = f(t)$ then $\Vdash^{V[G_{\MPB}]}_{\MQB}$``$\check{s} \perp_{\dot{T}} \check{t}$''.
\end{enumerate}
The ordering is reverse inclusion.
\end{definition}
\begin{lemma}
\label{properties of generalized specializing forcing}
Work in  $V[G_{\MPB}]$.
\begin{enumerate}
\item [(a)] The forcing notion  $\mathbb{B}_\MQB(\dot{T})$ is $\aleph_1$-closed.

\item [(b)] In the generic extension by $\mathbb{B}_\MQB(\dot{T})$, there is a function $F: \kappa \times \omega_1 \to \omega_1$
which is a specializing function  of every generic interpretation of $\dot{T}$ by a $\MQB$-generic filter over $V[G_{\MPB}]$.
\end{enumerate}
\end{lemma}
In general, $\mathbb{B}_\MQB(\dot{T})$  may fail to satisfy the $\kappa$-c.c.
However as we will see in the proof of Theorem \ref{main theorem2}, under some suitable assumptions, $\mathbb{B}_\MQB(\dot{T})$ will satisfy the $\kappa$-c.c., which is the crucial part of the argument.

\subsection{Definition of the main forcing}
\label{Definition of the main forcing}
In this subsection, we define our main forcing notion, which will be used in the proof of Theorem \ref{main theorem2}.
Assume that $\GCH$ holds and let $\kappa$ be a weakly compact cardinal. Let also $\delta > \kappa$ be a regular cardinal and fix a function $\Phi: \delta \rightarrow H(\delta)$ such that for each $x \in H(\delta), \Phi^{-1}(x)$ is unbounded in $\delta$.
\begin{remark}
For the proof of  Theorem \ref{main theorem2}, it suffices to take $\delta=\kappa^+$, but
we present a more general result that will be used for the proof of
Theorems \ref{main theorem3} and \ref{main theorem}
\end{remark}

We define by induction on $\alpha \leq \delta$  two iterations of forcing notions
\[
\MPB^1_{\delta}=\langle \langle \MPB^1_\alpha \mid \alpha \leq \delta  \rangle,  \langle \dot{\MQB}^1_\alpha \mid \alpha < \delta   \rangle \rangle
\]
and
\[
\MPB^2_{\delta}=\langle \langle \MPB^2_\alpha \mid \alpha \leq \delta  \rangle,  \langle \dot{\MQB}^2_\alpha \mid \alpha < \delta   \rangle \rangle.
\]

The superscript indicates on which cardinal trees are specialized: $\MPB^1_\delta$ is responsible for specialization of $\aleph_1$-Aronszajn trees while $\MPB^2_\delta$ will provide specialization functions for $\aleph_2$-Aronszajn trees. 
Suppose that $\alpha \leq \delta$ and we have defined the forcing notions $\MPB^1_\beta$
and $\MPB^2_\beta$ for all $\beta < \alpha$.
Let us define $\MPB^1_\alpha$ and $\MPB^2_\alpha$.

\subsubsection*{Definition of \texorpdfstring{$\MPB^2_\alpha$}{P2}} The forcing notion $\MPB^2_\alpha$ is defined in $V$ as follows.

Set $\MQB^2_0=\Col(\aleph_1, < \kappa)$.

If $\alpha$ is a limit ordinal and $\cf(\alpha) > \omega$, let $\MPB^2_\alpha$ be the direct limit of the forcing notions
$\MPB^2_\beta$, $\beta < \alpha$. If $\alpha$ is a limit ordinal and $\cf(\alpha) = \omega$, let $\MPB^2_\alpha$ be the inverse limit of the forcing notions
$\MPB^2_\beta$, $\beta < \alpha$.

Now suppose that $\alpha=\beta+1$ is a successor ordinal. If $\Phi(\beta)$ is a $\MPB^2_\beta \ast \dot{\MPB}^1_\beta$-name for a $\kappa$-Aronszajn tree,
then let $\MQB^2_\beta$ be a $\MPB^2_\beta$-name such that
\[
\Vdash_{\MPB^2_\beta}\text{``}\dot{\MQB}^2_\beta = \mathbb{B}_{\MPB^1_\beta}(\Phi(\beta))\text{''}.
\]
Otherwise, let  $\MQB^2_\beta$ be a name for the trivial forcing notion.

\subsubsection*{Definition of \texorpdfstring{$\MPB^1_\alpha$}{P1}} The forcing notion $\MPB^1_\alpha$ is defined in the generic extension of $V$ by $\MPB^2_\alpha$.
Let $V[G^2_\alpha]$ be the generic extension of $V$ by $\MPB^2_\alpha$ and work in it.

If $\alpha$ is a limit ordinal, then let $\MPB^1_\alpha$ be the direct limit of the forcing notions
$\MPB^1_\beta$, $\beta < \alpha$.

Let $\alpha=\beta+1$ be a successor ordinal. If $\Phi(\beta)$ is a $\MPB^2_\alpha \ast \dot{\MPB}^1_\beta$-name for an $\aleph_1$-Aronszajn tree,
then let $\MQB^1_\beta$ be such that
\[
\Vdash_{\MPB^2_\alpha \ast \dot{\MPB}^1_\beta}\text{``}\dot{\MQB}^1_\beta = \mathbb{B}(\Phi(\beta))\text{''}.
\]
Otherwise, let $\MQB^1_\beta$ be the trivial forcing notion.

\subsubsection*{Definition of the main forcing notion} Finally we define the main forcing notion that will be used in the proof of Theorem \ref{main theorem2}.
For each $\alpha \leq \delta$ set $\MPB_\alpha = \MPB^2_\alpha \ast \dot{\MPB}^1_\alpha$ and let $\MPB = \MPB_{\delta}$.

We will show that in the generic extension by $\MPB$, all Aronszajn trees on $\aleph_1$ and $\aleph_2$ are special, and there is an $\aleph_2$-Aronszajn tree.

It is important to note that although $\mathbb{P}^2_{\alpha}$ and $\mathbb{P}^1_{\alpha}$ are defined recursively together, $\mathbb{P}^2_{\alpha}$ does not depend on the generic filter of $\mathbb{P}^1_{\alpha}$ and  specializes any possible $\mathbb{P}^1_{\alpha}$-name for an $\aleph_2$-Aronszajn tree, regardless of whether this tree happened to be special or non-special in the generic extension by $\mathbb{P}^1_{\alpha}$ (see Lemma \ref{properties of generalized specializing forcing}(b)). 
\subsection{Properties of the forcing notion \texorpdfstring{$\MPB$}{P}}
\label{Properties of the forcing notion P}
In this subsection we state and prove some basic properties of the forcing notions defined above.
\begin{lemma}
\label{closure of P}
For every $\alpha \leq \delta$, the  forcing notion $\MPB^2_{\a}$ is $\aleph_1$-closed.
\end{lemma}
\begin{proof}
$\MPB^2_{\a}$ is a countable support iteration of $\aleph_1$-closed forcing notions, and hence is $\aleph_1$-closed.
\end{proof}
Then next lemma resembles Lemma \ref{basic properties of B(T)}.
\begin{lemma}
\label{chain condition of P-1}
For every $\alpha \leq \delta$, $\Vdash_{\MPB^2_\alpha}$`` $\dot{\MPB}^1_\alpha$ is c.c.c.''. Moreover, 
for every $\alpha \leq \delta$, $\Vdash_{\MPB^2_{\delta}}$`` $\dot{\MPB}^1_\alpha$ is c.c.c.''.
\end{lemma}
\begin{proof}
Let us show, by induction on $\alpha \leq \delta$, that $\MPB^1_\alpha$ is c.c.c.\ in the generic extension by $\MPB^2_\gamma$, for all $\gamma \in [\alpha, \delta]$.

For a limit ordinal $\alpha$, $\MPB^1_\alpha$ is the direct limit of the forcing notions $\MPB_{\beta}^1$, $\beta < \alpha$, and thus it is c.c.c.

Let $\alpha = \beta + 1$ be a successor ordinal.

Then either $\MPB^1_\alpha = \MPB^1_{\beta}$ and there is nothing to prove, or else, $\MPB^1_\alpha = \MPB^1_{\beta} \ast \mathbb{B}(\dot{T})$ where $\dot{T}=\Phi(\beta)$ is a $\MPB^2_\alpha \ast \MPB^1_\beta$-name for an $\aleph_1$-Aronszajn tree. We need to show that the forcing $\mathbb{B}(\dot{T})$ is c.c.c. in the generic extension by $\MPB^2_{\gamma}$, for $\gamma \in [\alpha, \delta]$.
Since the conditions in Baumgartner's forcing are finite, this forcing is absolute between any model of set theory that contains the evaluation of the name $\dot{T}$. Thus, it is sufficient to show that the tree $T=\dot{T}[G_{\MPB^2_\alpha \ast \MPB^1_\beta}]$, which is Aronszajn in the generic extension by $\MPB^2_\alpha \ast \MPB^1_\beta$, remains Aronszajn in the generic extension by $\MPB^2_{\gamma} \ast \MPB^1_\beta$, for every $\gamma \in [\alpha, \delta]$.

Work in the generic extension by $\MPB^2_{\alpha}$ and let $\gamma \in [\alpha, \delta]$. In this model the tree $T$ is introduced by the forcing $\MPB^1_{\beta}$, which is c.c.c.\ (by the inductive assumption). Let $\mathbb{R}$ be the quotient forcing $\MPB^2_{\gamma} / \MPB^2_\alpha$. This forcing is $\aleph_1$-closed in the generic extension by $\MPB^2_\alpha$, as a countable support iteration of $\aleph_1$-closed forcing notions. By the induction hypothesis, $\MPB^1_\alpha$ is c.c.c.\  in the generic extension by $\MPB^2_\alpha$. Thus, we can apply \cite[Lemma 6]{Unger2012} over the generic extension by $\MPB^2_\alpha$, and conclude that forcing with $\mathbb{R}$ over the larger generic extension by $\MPB^2_\alpha \ast \MPB^1_\alpha$ does not introduce new branches to the $\aleph_1$-tree $T$. The lemma follows.
\end{proof}
The next lemma is the main step towards completing the proof of Theorem \ref{main theorem2}.
\begin{lemma}
\label{chain condition lemma of p-2}
$\MPB^2_{\alpha}$ is $\kappa$-Knaster for each $\alpha \leq \delta$. In particular,  $\MPB^2_{\delta}$ satisfies the $\kappa$-c.c.
\end{lemma}
Before we dive into the details, let us sketch the main ideas of the proof.

The proof consists of two steps. First, we will show that for every $\kappa$-Aronszajn tree $T$, that appears in the iteration, for many $\lambda < \kappa$, the relation between elements above the $\lambda$-th level of $T$ and elements below the $\lambda$-th level of the tree is undetermined by the restriction of the forcing to
some nicely chosen model
$\mathcal{M}_\lambda$ (we will make this statement more precise in the proof ahead). From this, we will conclude that for densely many conditions $p$ and for many $\lambda < \kappa$, there are extensions of $p$ into two stronger conditions $p', p''$, such that the restrictions of $p'$ and $p''$ to $\mathcal{M}_\lambda$ are the same, i.e., $p' \upharpoonright \mathcal{M}_\lambda= p'' \upharpoonright \mathcal{M}_\lambda$, and for every element $t$ in the domain of $p'$ or $p''$ above $\lambda$, $p'$ forces that $\sigma' \leq t$, $p''$ forces that $\sigma'' \leq t$ and $\sigma', \sigma''$ are incompatible. The witnesses $\sigma', \sigma''$, will depend also on $\MPB^1_{\delta}$.  We call $p'$ and $p''$  a separating pair for $p$.

The second step is, given a sequence of $\kappa$ many conditions in $\mathbb{P}^2_{\delta}$, $\langle p_i \mid i < \kappa\rangle$, to extend each $p_i$ to a separating pair $p_i', p_i''$ as above and then, using a $\Delta$-System argument, to fix the incompatibility witnesses in some diagonal way. Then, we will show that every $p_i$ and $p_j$ are compatible and in fact, $p_i' \cup p_j''$ is a condition.

The proof imitates the proof of Laver-Shelah's theorem for specializing all $\aleph_2$-Aronszajn trees \cite{laver-shelah}, but with one additional difficulty - the separating pairs in our construction deal also with the conditions in $\mathbb{P}^1_{\delta}$.

Let us now return to the course of the proof.
\begin{proof}
We prove by induction on $\beta \leq \delta$ that $\MPB^2_\beta$ satisfies the $\kappa$-Knaster property. It is clear that $\MPB^2_1 \simeq \Col(\aleph_1, < \kappa)$
is $\kappa$-Knaster. Now suppose that $\beta \leq \delta$ and each $\MPB^2_\a, \a < \beta$, is $\kappa$-Knaster. We show that
$\MPB^2_\beta$ is also $\kappa$-Knaster.

If $\cf(\beta) > \kappa$, then $\MPB^2_\beta$ is easily seen to be $\kappa$-Knaster, as any subset of $\MPB^2_{\beta}$ of size $\kappa$
is included in some $\MPB^2_\alpha$, for some $\alpha < \beta$, so, by the induction hypothesis, it contains a subset of size $\kappa$ of pairwise compatible elements in  $\MPB^2_{\a}$ and hence
each pair of elements in this subset will be compatible in  $\MPB^2_{\beta}$ as well.

Now suppose that $\cf(\beta) \leq \kappa$. 
Let $\theta > \delta$ be a sufficiently large regular cardinal and let $\mathcal{M} \prec H(\theta)$
be such that
\begin{itemize}
\item $|\mathcal{M}|=\kappa$ and $^{<\kappa}\mathcal{M} \subseteq \mathcal{M}$.
\item $V_{\kappa} \subseteq \mathcal{M}$.
\item $\kappa, \Phi, \delta, \beta, \langle \MPB^1_\alpha\mid \alpha\leq \delta \rangle,  \langle \MPB^2_\alpha\mid \alpha\leq \delta \rangle, \dots \in \mathcal{M}$.
\end{itemize}
Note that $\mathcal{M}$ computes correctly the cofinality of $\beta$, and contains some cofinal sequence, $\langle \beta_i \mid i < \cf(\beta)   \rangle\in\mathcal M$. Since $\cf(\beta) \leq \kappa \subseteq \mathcal{M}$, we have  also that 
$\{\beta_i \mid i < \cf(\beta)\} \subseteq \mathcal M$. 
 
Let $\bar{\mathcal{M}}$ be the transitive collapse of $\mathcal{M}$ with $\pi: \mathcal{M} \to \bar{\mathcal{M}}$
being the transitive collapse map. For each $x \in \mathcal{M}$ we write $x^*$ for $\pi(x)$.
Note that since $\kappa + 1 \subseteq \mathcal{M}$, for $A \subseteq \kappa$, $A \in \bar{\mathcal{M}}$,  if and only if $A \in \mathcal{M}$ and $A^* = A$.

By \cite{hauser}, there exists a transitive model $\mathcal{N}$, closed under $< \kappa$-sequences, and an elementary embedding
$j: \bar{\mathcal{M}} \to \mathcal{N}$ with critical point $\kappa$ such that $j, \bar{\mathcal{M}} \in \mathcal{N}$.
Let
\[
\mathcal{F} = \{ A \subseteq \kappa \mid A \in \mathcal{M} \text{~and~}\kappa \in j(A)\}.
\]
Then $\mathcal{F}$ is an $\mathcal{M}$-normal $\kappa$-complete $\mathcal{M}$-ultrafilter on $\kappa$.
Let also $\mathcal{S}$ be the collection of $\mathcal{F}$-positive sets, i.e.,
\[
\mathcal{S}=\{ D \subseteq \kappa \mid \forall A \in \mathcal{F},~ D \cap A \neq \emptyset \}.
\]
\begin{lemma}
Every member of $\mathcal{F}$ is positive with respect to the weakly compact filter.
\end{lemma}
\begin{proof}
Let $A \in \mathcal{F}$. If $A$ is disjoint from some element in the weakly compact filter, then since $\mathcal{M}\prec H(\theta)$, there is some element in the weakly compact filter $B \in \mathcal{M}$ disjoint from $A$. By the definition of the weakly compact filter, this means that there is some parameter  $R \subseteq V_\kappa$ and a $\Pi^1_1$-formula $\Phi$ such that $\langle V_\kappa, \in, R\rangle \models \Phi$ and $B \supseteq \{\alpha < \kappa \mid \langle V_\alpha, \in, R \cap V_\alpha\rangle \models \Phi\}$. By elementarity we may assume that $R \in \mathcal M$. Note also that the transitive collapse does not modify $B, R$ or $A$.

Let us consider $j(B)$. By the definition of $B$, $\kappa \in j(B)$, since $j(R) \cap V_\kappa = R$ and 
\[V \models \text{``} \langle V_\kappa, \in ,R\rangle \models \Phi \text{''},\] so in particular 
\[\mathcal N \models \text{``} \langle V_\kappa, \in ,R\rangle \models \Phi \text{''}.\]
By $\kappa \in j(A)$ so $j(A) \cap j(B) \neq \emptyset$ and thus $A \cap B \neq \emptyset$.
\end{proof}
Let us note that although $\mathcal{F}$ does not have to be $V$-normal, it is $\mathcal M$ normal. Thus, if $\langle B_\alpha \mid \alpha < \kappa\rangle$ is a $\kappa$-sequence of elements in the model $\mathcal{M}$, such that all of them are in $\mathcal F$, then their diagonal intersection is in $\mathcal F$. Moreover, if $\mathcal{M}$ is a $\kappa$-model and $\mathcal{M} \in \mathcal{M}^*$, a larger $\kappa$-model, and if $j^* \colon \mathcal{M}^* \to \mathcal{N}^*$ is a weakly compact embedding then $j = j^* \restriction \mathcal{M} \colon \mathcal{M} \to j^*(\mathcal{M})$ is also a weakly compact embedding, and if $\mathcal{F}^*$ is the $\mathcal{M}^*$-ultrafilter defined by $j^*$ and $\mathcal{F}$ is the $\mathcal M$-ultrafilter defined by $j$, then $\mathcal{F} = \mathcal{F}^* \cap \mathcal{M}$. In particular, if a sequence of sets in $\mathcal{M}$ are of measure one regardless of the choice of $j$ then we can safely assume that their diagonal intersection is also of measure one.  

Let us define the sequence $\langle \mathcal{M}_\lambda \mid \lambda < \kappa  \rangle$ as follows:
Let $\phi: V_\kappa \leftrightarrow \mathcal{M}$ be a bijection and for each $\lambda < \kappa$ set
$\mathcal{M}_\lambda = \phi[V_\lambda]$. By the above discussion, we can assume that 
\[\{\lambda < \kappa \mid \mathcal{M}_\lambda \cap \kappa = \lambda \} \in \mathcal{F}.\]

Since $\mathcal{M} \prec H(\theta)$, if $\MPB^2_\beta$ is not $\kappa$-Knaster, then
$\mathcal{M} \models$`` $\MPB^2_\beta\text{ is not }\kappa\text{-Knaster}$''. In particular, there is a sequence of conditions $\langle p_\alpha \mid \alpha < \kappa\rangle \in \mathcal{M}$ witnessing it and since $\kappa \subseteq \mathcal{M}$, $p_\alpha \in \mathcal{M}$ for all $\alpha < \kappa$. We conclude that $\MPB^2_\beta \cap \mathcal{M}$ is not $\kappa$-Knaster. Thus, let us concentrate in showing that $\MPB^2_\beta \cap \mathcal{M}$ is $\kappa$-Knaster.

Let us assume that $\lambda < \kappa$ is an inaccessible cardinal, $^{<\lambda}\mathcal{M}_\lambda \subseteq \mathcal{M}_\lambda$ and that $\mathbb{P}^2_\beta \cap \mathcal{M}_\lambda$ is a regular subforcing of $\mathbb{P}^2_\beta \cap \mathcal{M}$ (later in Claim \ref{separation claim}, we will show that such cardinals exist). For  such a cardinal $\lambda$ and $p\in \MPB^2_\beta \cap \mathcal{M}$, we denote by $p \restriction \mathcal{M}_\lambda$, the following condition in $\MPB^2_\beta \cap \mathcal{M}_\lambda$.
Let $p(\alpha)$ be the $\alpha$-th coordinate of $p$ for $\alpha < \beta\rangle$. Then $p \restriction \mathcal{M}_\lambda$ is a function such that $(p\restriction \mathcal{M}_\lambda)(\alpha)$ is the trivial condition if $\alpha \notin \mathcal{M}_\lambda$ and otherwise $(p\restriction \mathcal{M}_\lambda)(\alpha) = p(\alpha) \restriction \mathcal{M}_\lambda$. Namely, $p \restriction \mathcal{M}_\lambda$ is obtained from $p$ by removing all coordinates which do not appear in $\mathcal{M}_\lambda$ and restricting the domain of the specialization functions to values from $\mathcal{M}_\lambda$. By the closure of $\mathcal{M}_\lambda$, $p \restriction \mathcal{M}_\lambda \in \mathcal{M}_\lambda$. Under some closure assumptions on $\lambda$, $p \restriction \mathcal M_\lambda$ is a condition and $p \leq p \restriction \mathcal{M}_\lambda$. 

Let $\alpha < \beta$ and let us assume that $\MPB_\alpha \cap \mathcal{M}_\lambda$ is $\lambda$-c.c.\ and that $\mathcal{M}_\lambda$ is sufficient closed so that $p \restriction \mathcal{M}_\lambda$ is a condition for densely many $p\in \MPB_\alpha \cap \mathcal M$. Let $G \subseteq \mathbb{P}^2_\alpha\cap \mathcal{M}$ be a generic filter. Then in $V[G]$ there is a natural generic filter,
\[G \cap \mathcal{M}_\lambda := \{p\restriction \mathcal{M}_\lambda \mid p \in G\} = \{p\in G \mid p \in \MPB^2_\alpha\cap \mathcal{M}_\lambda \} \subseteq \mathbb{P}^2_\alpha\cap \mathcal{M}_\lambda.\]

The genericity of $G \cap \mathcal{M}_\lambda$ follows from the chain condition of $\MPB_\alpha \cap \mathcal{M}_\lambda$. Indeed, if there is a maximal antichain of condition in $\MPB_\alpha \cap \mathcal{M}_\lambda$ then it is a member of $\mathcal{M}_\lambda$ and therefore maximal in $\mathcal{M}$ as well. In general, those equations might fail and this filter might not be generic as the map $p \mapsto p \restriction \mathcal{M}_\lambda$ is not a projection. 

Let us denote, temporarily, the quotient forcing $\left(\MPB^2_\alpha\cap \mathcal{M}\right) / \left(G \cap \mathcal{M}_\lambda\right)$ by $\mathbb{R}$. It is possible that for a condition $p\in\MPB^2_\alpha \cap \mathcal{M}$, $p \restriction \mathcal{M}_\lambda \not\Vdash \text{``}p \in \mathbb{R}$''. Nevertheless, it is impossible that $p \restriction \mathcal{M}_\lambda \Vdash \text{``}p \notin \mathbb{R}$'', and thus there is an extension $q \leq p \restriction \mathcal{M}_\lambda$ for which $q \Vdash p \in \mathbb{R}$ or equivalently for every $r \leq q$ in $\MPB^2_\alpha \cap \mathcal{M}_\lambda$, $r$ is compatible with $p$. By modifying $p \restriction \mathcal{M}_\lambda$ we can ensure that $p \restriction \mathcal{M}_\lambda \Vdash p \in \mathbb{R}$. 

This situation is denoted by $\ast_\lambda(p, p \restriction \mathcal{M}_\lambda)$ in \cite{laver-shelah}. In this paper, we will say in this case that $p$ is \emph{$\lambda$-compatible}. Similarly to \cite{laver-shelah}, we need to show that the collection of $\lambda$-compatible conditions is large:
\begin{claim}
The set of all $\lambda$-compatible conditions is dense. Moreover, every condition $p$ has a condition $q \leq p$, such that $q$ is $\lambda$-compatible and $\dom p \setminus \mathcal{M}_\lambda = \dom q \setminus \mathcal{M}_\lambda$.
\end{claim}

Unfortunately, a limit of a countable decreasing sequence of $\lambda$-compatible conditions is not necessarily $\lambda$-compatible. 

Before diving into the main technical lemma, let us use the following analysis of names of branches in the trees $\Phi(\alpha)$.
\begin{notation}
For forcing notions $\MPB$ and $\MQB$, we use $\MPB \lessdot \MQB$ to mean that $\MPB$ is a regular sub-forcing of $\MQB$.
\end{notation}

\begin{claim}\label{claim: countably names for each element}
Let $\lambda < \kappa$ be an inaccessible cardinal such that:
\begin{enumerate}
\item $\mathcal{M}_\lambda \cap \kappa = \lambda$.
\item $^{<\lambda}\mathcal{M}_\lambda \subseteq \mathcal{M}_\lambda$.
\item For every $\alpha \in \mathcal{M}_\lambda \cap \beta$, $\MPB^2_\alpha \cap \mathcal{M}_\lambda \lessdot \MPB^2_\alpha \cap \mathcal{M}$ and is $\lambda$-c.c.
\end{enumerate}
Then for every $\alpha \in \mathcal{M}_\lambda \cap \beta$, $\Vdash_{\MPB^2_\alpha} \MPB^1_\alpha \cap \mathcal{M}_\lambda \lessdot \MPB^1_\alpha \cap \mathcal{M}$ and it is equivalent to a sub-iteration.

Moreover, every cofinal branch in $T_\alpha \cap (\lambda\times \omega_1)$ in $\MPB^2_\alpha \ast \MPB^1_\alpha$ exists in $\MPB^2_\alpha \ast (\MPB^1_\alpha \cap \mathcal{M}_\lambda)$.
\end{claim}
Note that in this lemma we consider all branches that were introduced by the full forcing $\MPB^2_\alpha \ast \MPB^1_\alpha$, and not only names with respect to $\left(\MPB^2_\alpha \ast \MPB^1_\alpha\right) \cap \mathcal{M}$. 
\begin{proof}
Let $I = \mathcal{M}_\lambda \cap \beta$. Using the closure of the model $\mathcal{M}_\lambda$, it is easy to verify that for each $\gamma\in I$, the name for the $\gamma$-th Aronszajn tree in the iteration of $\MPB^1_\alpha$ is equivalent to an $\MPB^1_{I \cap \gamma}$-name. Indeed, one can consider the canonical name for the $\gamma$-th Aronszajn tree and using the chain condition of the forcing, conclude that it is contained in $\mathcal{M}_\lambda$.  In particular, it mentions only elements that appear in the the coordinates from the set $I$. As in Lemma \ref{chain condition of P-1}, their Aronszajnity is preserved. 

The quotient $\left(\MPB^2_\alpha \ast \MPB^1_\alpha \right) / (\MPB^2_\alpha \ast (\MPB^1_\alpha \cap \mathcal{M}_\lambda))$ is a finite support iteration of Baumgartner's forcing, and in the generic extension by $\MPB^2_\alpha \ast (\MPB^1_\alpha \cap \mathcal{M}_\lambda)$, $\lambda$ has cofinality $\omega_1$. Thus, by Lemma \ref{lemma: not adding new branches by baumgartner forcing}, no new cofinal branch to $T_\alpha \cap (\lambda\times \omega)$ is added by this forcing.
\end{proof}
Since the forcing $\MPB^1_\alpha / (\MPB^1_\alpha \cap \mathcal{M}_\lambda)$ is c.c.c.\ in the generic extension by $\MPB^2_\alpha \ast (\MPB^1_\alpha \cap \mathcal{M}_\lambda)$, for a given name for a branch $\dot{b}$, one can find in the ground model countably many $\MPB^2_\alpha \ast (\MPB^1_\alpha \cap \mathcal{M}_\lambda)$-names $\{\dot{b}_n \mid n < \omega\}$ for branches,  such that the weakest condition of the quotient forcing, forces that $\dot{b}$ is evaluated as one of them. The same holds, using the same arguments, when replacing $\MPB_\alpha^2$ with $\MPB^2_\alpha\cap \mathcal{M}$.

The main technical tool is the following separation claim.
\begin{claim}
\label{separation claim}
Assume that for all $\alpha \in \mathcal{M} \cap \beta$, $\MPB^2_\alpha$ is $\kappa$-Knaster and $T_\alpha=\Phi(\alpha)$ is a $\MPB^2_\alpha \ast \dot{\MPB}^1_\alpha$-name for a
 $\kappa$-Aronszajn tree. For each $\alpha \in \mathcal{M} \cap \beta$ there exists a measure one set $B_\alpha \in \mathcal{F}$ such that
for every $\lambda \in B_\alpha$:
\begin{enumerate}
\item $\alpha \in \mathcal{M}_\lambda$.
\item $\mathcal{M}_\lambda \cap \kappa = \lambda$ and $\lambda$ is inaccessible.
\item $\mathcal{M}_\lambda$ is closed under $<\lambda$-sequences.
\item $\MPB^2_\alpha \cap \mathcal{M}_\lambda \lessdot \MPB^2_\alpha \cap \mathcal{M}$ and is $\lambda$-c.c.
\item $\MPB^1_\alpha \cap \mathcal{M}_\lambda$ is (equivalent to) an $\MPB^2_\alpha \cap \mathcal{M}_\lambda$-name. 
\item $\left(\MPB^2_\alpha \ast \MPB^1_\alpha \right)\cap \mathcal{M}_\lambda  \lessdot \left(\MPB^2_\alpha \ast \MPB^1_\alpha\right) \cap \mathcal{M}$. Moreover, $\MPB^1_\alpha \cap \mathcal{M}_\lambda$ is a sub-iteration of $\MPB^1_\alpha \cap \mathcal{M}$ in the generic extension by $\MPB^2_\alpha \cap \mathcal{M}$.
\item $\left(\MPB^2_\alpha \ast \MPB^1_\alpha \right)\cap \mathcal{M}_\lambda$ forces that $T_\alpha \cap (\lambda\times\omega_1)$ is an Aronszajn tree.
\end{enumerate}

For every such $\lambda$ we have:
\begin{enumerate}
\item [(8)] $\Vdash_{\MPB_\alpha^2} \MPB^1_\alpha \cap \mathcal{M}_\lambda \lessdot \MPB^1_\alpha$.
\item [(9)] For every pair of $(\MPB_\alpha^2 \cap \mathcal{M}) \ast (\MPB_\alpha^1 \cap \mathcal{M}_\lambda)$-names of cofinal branches $\dot{\tau}, \dot{\theta}$ in the first $\lambda$ levels of $T_\alpha$ and $p \in \MPB^2_\beta \cap \mathcal{M}_\lambda$, and for every $\lambda$-compatible $q', q'' \in \MPB^2_\beta \cap \mathcal{M}$ with $p = q' \upharpoonright  \mathcal{M}_\lambda = q'' \upharpoonright  \mathcal{M}_\lambda$,
there are $\lambda$-compatible conditions $p', p'' \in \MPB^2_\beta \cap \mathcal{M}$, and a countable sequence $\langle (\bar{p}_n, \xi_n, \theta_n, \tau_n) \mid n < \omega\rangle \in \mathcal{M}_\lambda$
such that:
\begin{enumerate}
\item $p' \leq q',\, p'' \leq q''$ and $p' \upharpoonright  \mathcal{M}_\lambda = p'' \upharpoonright  \mathcal{M}_\lambda$.
\item $\forall n < \omega$, $p' \upharpoonright  \mathcal{M}_\lambda  \Vdash_{\MPB^2_\beta \cap  \mathcal{M}_\lambda}$`` $\bar{p}_n \in \dot{\MPB}^1_\beta \cap \mathcal{M}_\lambda$''.
\item $\forall n < \omega$, $\xi_n < \lambda$, $\theta_n, \tau_n \in \{\xi_n\} \times \omega_1$  and $\theta_n \neq \tau_n$.
\item $\forall n < \omega$, $(p' \upharpoonright \alpha, \bar{p}_n \upharpoonright \alpha) \Vdash$``$~\check{\tau_n}  \leq_{T_\alpha} \check{\tau}$'' and $(p'' \upharpoonright \alpha, \bar{p}_n  \upharpoonright \alpha) \Vdash$``$~\check{\theta_n} \leq_{T_\alpha} \check{\theta}$''.
\item $p' \upharpoonright  \mathcal{M}_\lambda  \Vdash_{\MPB^2_\beta \cap  \mathcal{M}_\lambda}$``$\{\bar{p}_n \mid n < \omega\}$ is a maximal antichain in $\dot{\MPB}^1_\beta$''.
\end{enumerate}
\end{enumerate}

Moreover, there is a large set $B$ such that $\lambda \in B$ implies that $\lambda \in B_\alpha$ for all $\alpha \in M_\lambda$. For densely many conditions $p \in \MPB_{\beta} \cap \mathcal{M}$ and $\lambda \in B$, $p \restriction \mathcal{M}_\lambda$ is a condition.
\end{claim}
\begin{proof}
First note that the ``moreover'' part is an application of diagonal intersection: by taking a slightly larger model that contains $\mathcal M$, we may assume that the function $\alpha \to B_\alpha$ for $\alpha \in \mathcal M \cap \beta$ is in the model, so it follows from the first part. In order to conclude that the restriction of a condition $p$ to $\mathcal{M}_\lambda$ results in a condition let us note that for every $\alpha \in \mathcal{M}_\lambda \cap \beta$ and every $x, y \in \dot{T}_\alpha \restriction \mathcal{M}_\lambda$, their compatibility is decided by some maximal antichain which belongs (and contained) in $\mathcal{M}_\lambda$. Thus, by extending $p$ at coordinates in $\mathcal{M}_\lambda$ below $\alpha$   we get that $(p \cap \mathcal{M}_\lambda) \restriction \alpha$ already forces the required incompatibility. By repeating this process countably many times, and using the closure of the forcing, we obtained the required condition. Note that this process only modifies $p \restriction \mathcal{M}_\lambda$. 

By the hypotheses of Claim \ref{separation claim}, $\MPB^2_\alpha$ has the $\kappa$-c.c. Let $B_\alpha$ be the set of all inaccessible
cardinals  $\lambda < \kappa$ that satisfy the requirements $(1)$-$(6)$ of the lemma.

Let us verify that $\kappa \in j(B_\alpha)$, and hence
$B_\alpha \in \mathcal{F}$. First, note that since the sequence $\langle \mathcal{M}_\lambda \mid \lambda < \kappa\rangle$ is continuous, $j(\mathcal{M})_\kappa = \bigcup_{\lambda < \kappa} j(\mathcal{M}_\lambda) = j`` \mathcal{M}$.
\begin{enumerate}
\item $j(\alpha) \in j`` \mathcal{M}$, since $\alpha \in \mathcal{M}$ by the assumption of the lemma.
\item $j`` \mathcal{M} \cap j(\kappa) = \kappa$.
\item $j`` \mathcal{M}$ is closed under $<\kappa$-sequences. This is true since $\mathcal{M}$ is closed under $< \kappa$-sequences.
\item $j(\MPB^2_\alpha) \cap j`` \mathcal{M} = j``(\MPB^2_\alpha \cap \mathcal{M})$ and in particular, it is isomorphic to $\MPB^2_\alpha \cap \mathcal{M}$ and is $\kappa$-c.c. From this fact, together with the closure of $j`` \mathcal{M}$ we conclude that it is a regular subforcing of $j(\MPB^2_\alpha \cap \mathcal{M})$.
\item This is the same as in the previous assertion.
\item Using the previous item and the chain condition of the forcing.
\item As in the previous assertion, $j(\MPB^2_\alpha \ast \MPB^1_\alpha) \cap j`` \mathcal{M}$ is isomorphic to $(\MPB^2_\alpha \ast \MPB^1_\alpha) \cap \mathcal{M}$. By the chain condition of the forcing   $\MPB^2_\alpha \ast \MPB^1_\alpha$, $(\MPB^2_\alpha \ast \MPB^1_\alpha) \cap \mathcal{M} \lessdot \MPB^2_\alpha \ast \MPB^1_\alpha$. Thus, we conclude that $j`` T_\alpha$ which is exactly the name of $j(T_\alpha) \cap (\kappa \times \omega_1)$, is a name with respect to the regular subforcing $j``\left( (\MPB^2_\alpha \ast \MPB^1_\alpha) \cap \mathcal{M}\right)$. Clearly, the subforcing forces it to be an Aronszajn tree.
\end{enumerate}

Next, let us show that that the elements of $B_\alpha$ satisfy the clauses $(8)$ and $(9)$ of the lemma. $(8)$ follows from the closure of the model and the chain condition of the forcing.
Suppose that
$\lambda \in B_\alpha$, and  fix names  $\dot\theta$ and $\dot\tau$  for branches, and conditions $p, q'$ and $q''$ as in the statement of the lemma.

It then follows from the choice of $\lambda$ that, for any $(\MPB_\alpha^2 \cap \mathcal{M}) \ast (\MPB_\alpha^1 \cap \mathcal{M}_\lambda)$-generic filter $G$
over $V$, the branches $\dot{\theta}^G, \dot{\tau}^G \notin V[G_{(\MPB^2_\alpha \ast \MPB^1_\alpha) \cap \mathcal{M}_\lambda}]$,
where $G_{(\MPB^2_\alpha \ast \MPB^1_\alpha) \cap \mathcal{M}_\lambda}= G \cap \left((\MPB^2_\alpha \ast \MPB^1_\alpha) \cap \mathcal{M}_\lambda \right)$.

We now claim that below any pair of $\lambda$-compatible conditions $(p', \bar{p}), (p'', \bar{p}) \in \left( \MPB^2_\beta  \cap \mathcal{M} \right)\ast \left(\MPB^1_\beta \cap \mathcal{M}_\lambda\right)$,  $p'\upharpoonright \mathcal{M}_\lambda = p'' \upharpoonright \mathcal{M}_\lambda$, there is a pair of $\lambda$-compatible conditions $(q', \bar q) \leq (p', \bar{p}), (q'', \bar q) \leq (p'', \bar{p})$
such that $(q' \upharpoonright \alpha, \bar q \upharpoonright \alpha) , (q''\upharpoonright \alpha, \bar q \upharpoonright \alpha) \in \MPB^2_\alpha \ast \dot{\MPB}^1_\alpha$ force incompatible values for the branches below $\dot\theta$ and $\dot\tau$ and
 $q'\upharpoonright \mathcal{M}_\lambda = q'' \upharpoonright \mathcal{M}_\lambda$.

If not, we can find conditions $(p', \bar p), (p'', \bar p)$ so that for any extensions $q' \leq p'$
and $q'' \leq p''$ which are $\lambda$-compatible and $q' \upharpoonright \mathcal{M}_\lambda = q'' \upharpoonright \mathcal{M}_\lambda$, and any $\bar q \leq \bar p$, the conditions $(q' \upharpoonright \alpha, \bar q \upharpoonright \alpha), (q'' \upharpoonright \alpha, \bar q \upharpoonright \alpha) \in \MPB^2_\alpha \ast \dot{\MPB}^1_\alpha$ can not force incompatible values for the branches $\dot\theta$ and $\dot\tau$ (using the elementarity of $\mathcal{M}$).

Let $G = G_{(\MPB^2_\alpha \ast \MPB^1_\alpha) \cap \mathcal{M}_\lambda}$ be $V$-generic for ${(\MPB^2_\alpha \ast \MPB^1_\alpha) \cap \mathcal{M}_\lambda}$, and let $H_1, H_2$ be mutually generic filters for the forcing $\MPB^2_\alpha/ (\MPB^2_\alpha \cap \mathcal{M}_\lambda)$ over the model $V[G]$. Let us assume that $(p' \upharpoonright \alpha) \restriction \mathcal{M}_\lambda \in G$, $p'\upharpoonright \alpha \in H_1$ and $p''\upharpoonright \alpha \in H_2$.

By the assumption, $\dot{\tau}^{G \ast H_1} = \dot{\theta}^{G \ast H_2}$. In particular, \[\dot{\tau}^{G \ast H_1} \in V[G][H_1] \cap V[G][H_2],\] and by the mutual genericity of $H_1$ and $H_2$ --- it is in $V[G]$, which is impossible.

Thus we can find a pair of conditions in the iteration 
\[(p'_0, \bar p_0), (p''_0, \bar p_0) \in \left(\MPB^2_\beta \cap \mathcal{M}\right) \ast \left(\dot{\MPB}^1_\beta \cap \mathcal{M}_\lambda\right)\] with $p'_0 \leq p', p''_0 \leq p''$
and  $p'_0 \upharpoonright \mathcal{M}_\lambda = p''_0 \upharpoonright \mathcal{M}_\lambda$ together with $\xi_0 < \lambda$ and  elements in
$\theta_0, \tau_0 \in \{\xi_0\} \times \omega_1$ such that
\begin{itemize}
\item $(p'_0\upharpoonright \alpha, \bar p_0\upharpoonright \alpha) \Vdash$``$\check\theta_0 \in \dot\theta$''.
\item $(p''_0\upharpoonright \alpha, \bar p_0\upharpoonright \alpha) \Vdash$``$\check\tau_0 \in \dot\tau$''.
\end{itemize}
Let us repeat the process. Suppose that $\nu<\omega_1$
and we have defined  the pairs $(p'_\zeta, \bar p_\zeta), (p''_\zeta, \bar p_\zeta) \in \left(\MPB^2_\beta \cap \mathcal{M}\right) \ast \left(\dot{\MPB}^1_\beta \cap \mathcal{M}_\lambda\right)$ together with $\xi_\zeta$ and
$\theta_\zeta, \tau_\zeta \in \lambda \times \omega_1$ such that
\begin{itemize}
\item The sequences $\langle p'_\zeta \mid \zeta < \nu\rangle$ and $\langle p''_\zeta \mid \zeta < \nu \rangle$ are decreasing and for each $\zeta$, $p'_\zeta$ and $p''_\zeta$ are $\lambda$-compatible.
\item $p'_\zeta \upharpoonright \mathcal{M}_\lambda = p''_\zeta \upharpoonright \mathcal{M}_\lambda$.
\item $p'_\zeta \upharpoonright \mathcal{M}_\lambda \Vdash_{\MPB^2_\beta \cap \mathcal{M}_\lambda}$`` $\bar{p}_\zeta \in \dot{\MPB}^1_\beta \cap \mathcal{M}_\lambda$''.
\item For $\zeta'<\zeta < \nu$, $p'_\zeta \upharpoonright \mathcal{M}_\lambda \Vdash_{\MPB^2_\beta \cap \mathcal{M}_\lambda}$``$\bar p_{\zeta'}$ and $\bar p_\zeta$ are incompatible''.
\item $\xi_\zeta < \lambda, \theta_\zeta, \tau_\zeta \in \{\xi_\zeta\} \times \omega_1$ and $\theta_\zeta \neq \tau_\zeta$.
\item $(p'_\zeta \restriction \alpha, \bar p_\zeta \restriction \alpha) \Vdash$``$\check\theta_\zeta \in \dot\theta$''.
\item $(p''_\zeta \restriction \alpha, \bar p_\zeta \restriction \alpha) \Vdash$``$\check\tau_\zeta \in \dot\tau$''.
\end{itemize}
Let $q'_\nu=\bigcup_{\zeta<\nu}p'_\zeta$ and
$q''_\nu=\bigcup_{\zeta<\nu}p''_\zeta$. Then $q'_\nu, q''_\nu \in \MPB^2_\beta \cap \mathcal{M}$,
 $q'_\nu \upharpoonright \mathcal{M}_\lambda = q''_\nu \upharpoonright \mathcal{M}_\lambda$ and they are $\lambda$-compatible. If
\begin{center}
$q'_\nu \upharpoonright \mathcal{M}_\lambda \Vdash$``$\{\bar p_\zeta \mid \zeta < \nu\}$ is a maximal antichain'',
\end{center}
then we stop the construction. Otherwise find a condition $\bar q_\nu$ which is forced to be incompatible with all $\bar p_\zeta$'s, $\zeta<\nu$, and let $(p'_\nu, \bar p_\nu), (p''_\nu, \bar p_\nu)$, $\xi_\nu < \lambda$
and $\theta_\nu, \tau_\nu \in \{\xi_\nu\} \times \omega_1$
be such that
\begin{itemize}
\item $(p'_\nu, \bar p_\nu), (p''_\nu, \bar p_\nu) \in \left(\MPB^2_\beta \cap \mathcal{M}\right) \ast \left(\dot{\MPB}^1_\beta \cap \mathcal{M}_\lambda\right)$.
\item $(p'_\nu, \bar p_\nu) \leq (\bar q'_\nu, \bar q_\nu)$ and $(p''_\nu, \bar p_\nu) \leq (\bar q''_\nu, \bar q_\nu)$.
\item $p'_\nu \upharpoonright \mathcal{M}_\lambda= p''_\nu \upharpoonright \mathcal{M}_\lambda$.
\item $(p'_\nu \restriction \alpha, \bar p_\nu \restriction \alpha) \Vdash$``$\check\theta_\nu \in \dot\theta$''.
\item $(p''_\nu \restriction \alpha, \bar p_\nu \restriction \alpha) \Vdash$``$\check\tau_\nu \in \dot\tau$''.
\end{itemize}
We would like to claim that there is a way to construct this sequence in a way that the process terminates after at most countably many steps. Otherwise, for every countable $\nu$ and a choice for the values of $p'_\eta,p''_\eta,\bar{p}_\eta$ for $\eta < \nu$ there is a choice for $p'_\nu, p''_\nu, \bar{p}_\nu$. Let $H$ be a generic filter for $\MPB^2_\alpha \cap \mathcal{M}_\lambda$, and using the assumed density, find an $\omega_1$-sequence of conditions $p'_\nu, p''_\nu, \bar{p}_\nu$ such that $p'_\nu \restriction \mathcal{M}_\lambda = p''_\nu \restriction \mathcal{M}_\lambda \in H$. Note that in the generic extension by $H$, the sequence $\bar{p}_\nu$, $\nu < \omega_1$ is an antichain in $\MPB_\alpha^1$, which is a contradiction to Lemma \ref{chain condition of P-1}. Thus, the process generically terminates. Using the closure of $\MPB_\alpha^2$, Lemma \ref{closure of P}, we conclude that there is a choice of conditions such that this process terminates after at most countably many steps. At the end of the process, we get a countable ordinal $\vartheta$, sequences $\langle p_\zeta' \mid \zeta < \vartheta \rangle$
and $\langle p_\zeta'' \mid \zeta< \vartheta\rangle$ of conditions in $\MPB_\beta \cap \mathcal{M}$,
and sequences $\{\bar p_\zeta \mid \zeta < \vartheta\}$
and $\langle (\xi_\zeta, \theta_\zeta, \tau_\zeta) \mid \zeta < \vartheta\rangle$
such that
\begin{itemize}
\item The sequences $\langle p_\zeta' \mid \zeta < \vartheta \rangle$
and $\langle p_\zeta'' \mid \zeta < \vartheta \rangle$ are decreasing and $p_\zeta' \upharpoonright \mathcal{M}_\lambda = p_\zeta'' \upharpoonright \mathcal{M}_\lambda$.
Let $p'=\bigcup_{\zeta < \vartheta}p'_\zeta$ and $p''=\bigcup_{\zeta < \vartheta}p''_\zeta$.
\item $p' \upharpoonright \mathcal{M}_\lambda \Vdash_{\MPB^2_\beta\cap \mathcal{M}_\lambda}$``$\{\bar{p}_\zeta \mid \zeta < \omega\}$ is a maximal antichain in $\dot{\MPB}^1_\beta$''.
\item For all $\zeta < \vartheta, ~ \theta_\zeta, \tau_\zeta \in \{\xi_n\} \times \omega_1$
and $\theta_\zeta \neq \tau_\zeta$.

\item For all $\zeta < \vartheta$, $(p'_\zeta \upharpoonright \alpha, \bar{p}_\zeta \upharpoonright \alpha) \Vdash$``$~\check{\tau_\zeta} \in \dot{\tau}$'' and $(p''_\zeta \upharpoonright \alpha, \bar{p}_\zeta  \upharpoonright \alpha) \Vdash$``$~\check{\theta_\zeta} \in \dot{\theta}$''.
\end{itemize}
 Then $p', p''$ together with the sequence
$\langle (\bar p_\zeta, \xi_\zeta, \theta_\zeta, \tau_\zeta) \mid \zeta < \vartheta \rangle$
are as required.
\end{proof}
By reordering, we may squeeze the information to an $\omega$-sequence. Let us call the sequence
$\langle (\bar p_n, \xi_n, \theta_n, \tau_n) \mid n < \omega\rangle$ a \emph{$\lambda$-separating witness} for the branches $\theta, \tau$
relative to $p', p''$.

Let $\lambda$ be as in the claim. Let $p', p'' \in \MPB^2_\beta \cap \mathcal{M}$ be arbitrary $\lambda$-compatible conditions, with $p' \restriction \mathcal{M}_\lambda = p''\restriction \mathcal{M}_\lambda$. For every $\alpha \in \dom (p') \cap \mathcal{M}_\lambda$ and every element $\theta \in \dom(p'(\alpha))$ above $\lambda$, there are at most countably $\left(\MPB^2_\alpha \cap \mathcal{M}\right) \ast \left(\dot{\MPB}^1_\alpha \cap \mathcal{M}_\lambda \right)$-names for the branch $\{t \in T_\alpha \mid t \leq \theta,\, \Lev_{T_\alpha}(t) < \lambda\}$, by the Claim \ref{claim: countably names for each element} and the discussion following it.

If $\theta, \tau$ are elements in the tree $T_\alpha$ with $\Lev_{T_\alpha}(\theta), \Lev_{T_\alpha}(\tau) \geq \lambda$, then we may apply Claim \ref{separation claim} for the countably many possible pairs of names for the branches below $\lambda$ that $\theta$ and $\tau$ contributes, and obtain countably many separation pairs. Let us call this countable collection of separating witness, a \emph{$\lambda$-separating witness} for $\theta$ and $\tau$.

Let $B\in\mathcal{F}$ be as in the conclusion of Claim \ref{separation claim}. By a repeated usage of the conclusion of Claim \ref{separation claim}, for every condition $p \in \MPB^2_\beta \cap \mathcal{M}$, and every $\lambda \in B$, $\mathcal{M}_\lambda \cap \kappa=\lambda$, and there are $\lambda$-compatible conditions $p', p''\leq p$
such that
 \begin{itemize}
 \item $p' \upharpoonright \mathcal{M}_\lambda = p'' \upharpoonright \mathcal{M}_\lambda$.
 \item For every $\alpha \in \beta \cap \mathcal{M}_\lambda$, any pair of elements above $\lambda$ in $\dom(p'(\alpha))\times \dom(p''(\alpha))$ has a $\lambda$-separating witness in $\mathcal{M}_\lambda$ relative to $p'\restriction \a, p''\restriction \a$.
 \end{itemize}
We call this pair $(p', p'')$ a \emph{$\lambda$-separating pair}. In order to make sure that the process of producing a $\lambda$-separating pair converges, we used the fact that it is possible to restore $\lambda$-compatibility without changing the domain of the condition outside of $\mathcal{M}_\lambda$. Note that $\dom(p)$ might contain elements from $\mathcal{M} \setminus \mathcal{M}_\lambda$ which are not treated.

Now let $\langle p_\lambda \mid \lambda < \kappa\rangle \in \mathcal{M}$ be a sequence of conditions in $\MPB^2_\beta$.\footnote{Recall that if there exists a sequence of conditions $\langle p_\lambda \mid \lambda < \kappa\rangle$ which contradicts the $\kappa$-Knaster property of $\MPB^2_\beta$, then there is such a sequence in $\mathcal{M}$ as well, by elementarity.}

For every $\lambda\in B$, $p_\lambda$ can be extended to a $\lambda$-separating pair $(p'_\lambda, p''_\lambda) \in \mathcal{M}$. Let $s_\lambda \in \mathcal{M}_\lambda$ be the list of separating witnesses.

The function that sends $\lambda$ to $(s_\lambda, p' \restriction \mathcal{M}_\lambda)$ is regressive. By the normality of the weakly compact filter (recall that every member of $\mathcal{F}$ is positive with respect to the weakly compact filter), and by further shrinking if necessary, we may assume that there are $s_*, p_*$ such that on a positive set $D \in \mathcal{S}$,  for all $\lambda \in D$, $s_* = s_\lambda$ and $p_* = p' \restriction \mathcal{M}_\lambda$.

Moreover, by intersecting $D$ with a club, we may assume that for every $\lambda \in D$ and $\lambda' \in D$ above $\lambda$, $p'_\lambda, p''_\lambda  \in \mathcal{M}_{\lambda'}$ (in particular, the domain of $p'_\lambda, p''_\lambda$ as well as the domain of $p'_\lambda(\alpha), p''_\lambda(\alpha)$ are subsets of $\mathcal{M}_{\lambda'}$). By additional shrinking of $D$, if needed, we may assume that the sets $\supp(p'_\lambda) \cup \supp(p''_\lambda)$ form a $\Delta$-system with a root $\Lambda$, in the sense that, if $\lambda < \lambda'$ are in $D$, and $\alpha \in \left(\supp(p'_\lambda) \cup \supp(p''_\lambda)\right) \setminus \Lambda$, then $\alpha \notin \supp(p'_{\lambda'})
 \cup \supp(p''_{\lambda'})$. Without loss of generality, $\Lambda \subseteq \mathcal{M}_\lambda$ for $\lambda = \min D$.

We claim that for any $\lambda < \lambda'$ in $D$, $p_\lambda$ is compatible with $p_{\lambda'}$, and moreover this compatibility is witnessed by the condition $q$, which is defined by $q(\a) = p'_\lambda(\a) \cup p''_{\lambda'}(\a)$ for every $\a < \beta$.
It is enough to show that $q$ is a condition. Clearly, $\dom(q)$ is at most countable. Therefore, it is enough to show that $q\restriction \gamma$ forces that $q(\gamma)$ is a condition for all $\gamma < \beta$.
We prove this by induction on $\gamma < \beta$.

For $\gamma = 0$, $q(0)\in \Col(\aleph_1, <\kappa)$, since it is the union of two conditions that have the same intersection with $\mathcal{M}_\lambda$, and have disjoint domains above it.

Assume that $q\restriction \gamma$ is a condition.
We may assume that $T_\gamma=\Phi(\gamma)$ is a $\MPB^2_\gamma \ast \dot{\MPB}^1_\gamma$-name for a $\kappa$-Aronszajn tree, as otherwise the forcing at stage $\gamma$ is trivial. We may also assume that $\gamma \in \Lambda$, since otherwise either $\gamma \notin \supp(p'_\lambda)$ or $\gamma \notin \supp(p''_{\lambda'})$.
In order to show that $q \restriction \gamma \Vdash$`` $q (\gamma)$ is a condition'', we have to show that if $t, t' \in \dom(q(\gamma))$
 and $q(\gamma)(t)=q(\gamma)(t')$, then $q\restriction \gamma \Vdash_{\MPB^2_\gamma}$``$ 1_{\mathbb{P}^1_\gamma} \Vdash_{\mathbb{P}^1_\gamma} \check{t} \perp_{T_\gamma} \check{t'}$''.

We may suppose that both of $t$ and $t'$ are above $\lambda$, as otherwise we can use the fact $p'_\lambda \upharpoonright \mathcal{M}_\lambda = p''_{\lambda'} \upharpoonright \mathcal{M}_\lambda'$ and the fact that $\gamma \in \mathcal{M}_\lambda$,
 to conclude the result.

Recall that $(p'_\lambda, p''_\lambda)$ is a separating pair. Let $\dot{b}_t$ be one of the countably many possible names for branches below $\lambda$ of elements below $t$ and let $\dot{b}_{t'}$ be a corresponding name for $t'$. The separating witness $\langle \bar{p}_n, \tau_n, \theta_n \mid n < \omega\rangle$ was stabilized for elements in $D$, and thus $(p'_\lambda \restriction \gamma, \bar p_n \restriction \gamma) \Vdash$``$\check{\tau_n} \in \dot{b}_t$'' and $(p''_\lambda \restriction \gamma, \bar p_n \restriction \gamma) \Vdash$``$ \check{\theta}_n \in \dot{b}_{t^\prime}$'', where $\tau_n \neq \theta_n$.
By the induction hypothesis, $q \restriction \gamma$ is a condition and it is stronger than $p'_\lambda\restriction \gamma$ and $p''_{\lambda'} \restriction \gamma$. Let us denote, temporarily by $\tilde{t}$ the element in the $\lambda$-th level of $T_\gamma$ above $\dot{b}_t$ and by $\tilde{t}'$ the element in the $\lambda'$-th level of $T_\gamma$ above $\dot{b}_{t'}$.
We obtained that for all $n < \omega$, $(q \restriction \gamma, \bar p_n) \Vdash$``$\tilde{t}\perp_{T_\gamma} \tilde{t}'$ ''. Now if $q\restriction \gamma \not\Vdash_{\MPB^2_\gamma}$``$1_{\MPB^1_\gamma}\Vdash_{\mathbb{P}^1_\gamma} \check{t}\perp \check{t}' "$, then there is a condition $q^\prime \leq q\restriction \gamma$ and $\bar p\in \mathbb{P}^1_\gamma$ such that $(q^\prime, \bar p)\Vdash \tilde{t} = \tilde{t}^\prime$. But $\bar p$ is compatible with $\bar p_n$, for some $n < \omega$. As  $q^\prime$ is stronger than $q \restriction \gamma$,
 $(q^\prime, \bar p_n \restriction \gamma)\Vdash$``$ \tilde{t} =\tilde{t}^\prime$'',
it follows that
 \[
 (q^\prime, \bar p_n \restriction \gamma)\Vdash \text{~``~} \check{\theta}_n \leq_{T_\gamma} \tilde{t}^\prime=\tilde{t}\text{~''~} ~~~\&~~~~ (q^\prime, \bar p_n \restriction \gamma)\Vdash \text{~``~} \check{\tau}_n \leq_{T_\gamma} \tilde{t}\text{~''.~}
 \]
This is in contradiction with the choice of $\theta_n$ and $\tau_n$. Since this is true for all possible $\tilde{t} \leq_{T_\gamma} t$ and $\tilde{t'} \leq_{T_\gamma} t'$, we conclude that they are forced to be incompatible.

If $\gamma$ is a limit ordinal and $q \restriction \bar\gamma$ is a condition for all
$\bar\gamma < \gamma$, then $q \restriction \gamma$ is a condition as well.
Lemma \ref{chain condition lemma of p-2} follows.
\end{proof}

The next lemma follows from Lemmas \ref{chain condition of P-1}  and \ref{chain condition lemma of p-2}
\begin{lemma}
\label{chain condition for the main forcing}
For every $\alpha \leq \delta$, $\MPB^2_\alpha \ast \dot{\MPB}^1_\alpha$ satisfies the $\kappa$-c.c. In particular $\MPB= \MPB^2_{\delta} \ast \dot{\MPB}^1_{\delta}$ satisfies the $\kappa$-c.c.
\end{lemma}
Putting the above lemmas together, we obtain the following result.
\begin{lemma}
\label{summarizing cardinal structure in p extensions}
Suppose $G$ is $\MPB$-generic over $V$. Then
\begin{enumerate}
\item [(a)] $\aleph_1^{V[G]}=\aleph_1$, $\aleph_2^{V[G]}=\kappa$ and $\aleph_3^{V[G]}=\kappa^+$.

\item [(b)] $V[G] \models$``$2^{\aleph_0}=2^{\aleph_1}=\delta$''.
\end{enumerate}
\end{lemma}

\subsection{Completing the proof of Theorem \ref{main theorem2}.}
\label{Completing the proof of Theorem main theorem2}
In this subsection we complete the proof of Theorem \ref{main theorem2}. The next lemma follows from Lemma \ref{chain condition for the main forcing}
\begin{lemma}
Suppose $X \in V[G_{\MPB}]$ and $X \subseteq \kappa$. Then  $X \in V[G_{\MPB^2_\alpha \ast \MPB^1_\alpha}]$,  for some $\alpha < \delta$.
\end{lemma}
We start by showing that the special Aronszajn tree property holds in $V[G_{\MPB}]$.
\begin{lemma}
\label{specializing N-1 trees}
$\MPB$ forces $\TSP(\aleph_1)$.
\end{lemma}
\begin{proof}
Let $T$ be an  $\aleph_1$-Aronszajn tree and let  $\dot{T}$ be a $\MPB$-name for it. Then for some $\alpha < \delta$ it is in fact a $\MPB_\alpha$-name and $\dot{T}=\Phi(\alpha)$.
Then
\[
\Vdash_{\MPB^2_{\alpha+1} \ast \dot{\MPB}^1_{\alpha+1}} \text{``}\dot{T} \text{~is specialized'',~}
\]
and hence there exists $F \in V[G_{\MPB^2_{\alpha+1} \ast \dot{\MPB}^1_{\alpha+1}}]$ which is a specializing function for $T$.
As $V[G_{\MPB}] \supseteq V[G_{\MPB^2_{\alpha+1} \ast \dot{\MPB}^1_{\alpha+1}}]$ and these models have the same cardinals,  $F$ is also a specializing function for $T$
in $V[G_{\MPB}]$.
\end{proof}
In order to show that the forcing notion $\MPB$ specializes all $\aleph_2$-Aronszajn trees, we need the following lemma which is an analogue of
Lemma \ref{properties of generalized specializing forcing}(b).

\begin{lemma}
\label{existence of specializing functions}
Suppose $\alpha < \delta$ and $\Phi(\alpha)$ is a $\MPB^2_\alpha \ast \dot{\MPB}^1_\alpha$-name for a $\kappa$-Aronszajn tree. Then
in the extension by
$\MPB^2_{\alpha+1}$, there exists a function $F: \kappa \times \omega_1 \to \omega_1$
which is a specializing function  of every generic interpretation of $\Phi(\alpha)$ by a $\MPB^1_{\alpha}$-generic filter.
\end{lemma}

\begin{lemma}
\label{specializing N-2 trees}
$\MPB$ forces $\TSP(\kappa)$.
\end{lemma}
\begin{proof}
First, there is an $\aleph_2$-Aronszajn tree in the generic extension, as the forcing $\Col(\omega_1, < \kappa)$ adds a special $\aleph_2$-Aronszajn tree and cardinals are preserved in the rest of the iteration.

Let $T$ be a  $\kappa$-Aronszajn tree and let $\dot{T}$ be a $\MPB$-name for it. Then for some $\alpha < \delta$ it is in fact a $\MPB_\alpha$-name and $\dot{T}=\Phi(\alpha)$.
By Lemma \ref{existence of specializing functions}, there exists $F \in V[G_{\MPB^2_{\alpha+1} \ast \dot{\MPB}^1_\alpha}]$ which specializes $T$.
As $V[G_{\MPB}]$ is a cardinal preserving extension of $V[G_{\MPB^2_{\alpha+1} \ast \dot{\MPB}^1_\alpha}]$, $F$ also witnesses that $T$
is  specialized in $V[G_{\MPB}]$.
The lemma follows.
\end{proof}

\section{Specializing  names of higher Aronszajn trees: An abstract approach}
\label{section:abstract}
Let us note that in the  proof of Theorem \ref{main theorem2}, we did not use the way the forcing notions $\MPB^1_\alpha, \alpha \leq \delta$ were defined, but only the fact that they satisfy the $\text{c.c.c.}$ and that the forcing notions $\MPB^1_\alpha, \a \leq \delta$, do not add new branches to trees of height $\aleph_1$. In this section we present the above situation in an abstract way that will be used for the next sections of this paper.

Thus suppose that $\mu < \kappa < \delta$ are regular cardinals. Let $\Phi$ and $ \Psi$
be two functions such that:
\begin{itemize}
\item  $\Phi: \delta \to H(\delta)$ is such that for each $x \in H(\delta), \Phi^{-1}(x)$ is unbounded in $\delta$.
\item $\Psi: \delta \to H(\delta)$ is such that for each $\a<\delta, \Psi(\a)$ is a forcing notion.
\end{itemize}
Let
\[
\langle \langle \MPB^2_\alpha \mid \alpha \leq \delta\rangle, \langle \dot{\MQB}^2_\alpha \mid \alpha < \delta   \rangle \rangle.
\]
be a forcing iteration of length $\delta$, defined as follows:

Set $\MQB^2_0=\Col(\mu, < \kappa)$.

If $\alpha$ is a limit ordinal and $\cf(\alpha) \geq \mu$, let $\MPB^2_\alpha$ be the direct limit of the forcing notions
$\MPB^2_\beta, \beta < \alpha$. If $\alpha$ is a limit ordinal and $\cf(\alpha) < \mu$, let $\MPB^2_\alpha$ be the inverse limit of the forcing notions
$\MPB^2_\beta, \beta < \alpha$.

Now suppose that $\alpha=\beta+1$ is a successor ordinal. Let us assume that $\Psi(\beta)$ is such that $\Psi(\beta)=\MPB^2_\beta \ast \dot{\MPB}^1_\beta$ for some
$\MPB^2_\beta$-name $\dot{\MPB}^1_\beta$, where $\MPB^1_\beta$ is an iteration of length $\leq \beta$ with $<\zeta$-supports, for some $\zeta<\mu$, of forcing notions of size $<\kappa$. Moreover, let us assume that $\Phi(\beta)$ is a $\Psi(\beta)$-name for a $\kappa$-Aronszajn tree with the universe $\kappa \times \mu$.
Then let $\dot{\MQB}^2_\beta$ be a $\MPB^2_\beta$-name, such that in the generic extension $V[G_{\MPB^2_\beta}]$, the forcing notion
$\MQB^2_\beta$ is defined as follows:
\begin{itemize}
\item Conditions in $\MQB^2_\beta$ are partial functions $f: \kappa \times \mu \to \mu$ such that:
\begin{enumerate}
\item $\dom(f) \subseteq \kappa \times \mu$ has size $< \mu$.
\item If $s, t \in \dom(f)$ and $f(s) = f(t)$ then $\Vdash_{\MPB^1_\beta}$``$\check{s} \perp_{\Phi(\beta)} \check{t}$''.
\end{enumerate}
\item For $f, g \in \MQB^2_\beta, f \leq g$ if and only if $f \supseteq g$.
\end{itemize}
Otherwise, let  $\dot\MQB^2_\beta$ be a name for the trivial forcing notion.

It is obvious that the forcing notions $\MPB^2_\a, \a \leq \delta$ are $\mu$-directed closed.

Let us recall all of those properties which were used in the proof of Claim \ref{separation claim}.
\begin{definition}
We say that the triple $(\Phi, \Psi, \delta)$ is $(\mu, \kappa)$-suitable, if the following conditions hold. First, let $\MPB^2_\alpha, \dot{\MPB}^1_\alpha$, be defined as above using $\Phi$ and $\Psi$. Also, let $\langle \mathcal{M}_\lambda \mid \lambda < \kappa\rangle$ be a continuous chain of elementary submodels of the universe of size $<\kappa$ which contain all the relevant information. Let $\mathcal{M} = \bigcup_{\lambda < \kappa} \mathcal{M}_\lambda$.
\begin{enumerate}
\item $\mu < \kappa$ are regular cardinals and $\Phi, \Psi \colon \delta \to H(\delta)$ are as above.
\item For each $\a \leq \delta$ and $\gamma \in [\a, \delta], \Vdash_{\MPB^2_\gamma}$``$\dot{\MPB}^1_\a$ is $\mu$-c.c.''.
\item For each $\lambda < \kappa$ and $\a \in \mathcal{M}_\lambda \cap \delta$, if
\begin{enumerate}
\item $\MPB^2_\a \cap \mathcal{M}_\lambda \lessdot \MPB^2_\a \cap \mathcal{M}$.
\item $\Vdash_{\MPB^2_\a \cap \mathcal{M}}$`` $ \MPB^1_\a \cap \mathcal{M}_\lambda \lessdot \MPB^1_\a \cap \mathcal{M}$ and moreover, it is a sub-iteration''.
\item $\Phi(\a)$ is a $\MPB^2_\a \ast \dot{\MPB}^1_\a$-name for a $\kappa$-Aronszajn tree.
\item $\Phi(\a) \cap \mathcal{M}_\lambda$ is a $(\MPB^2_\a \cap \mathcal{M}_\lambda) \ast (\dot{\MPB}^1_\alpha \cap \mathcal{M}_\lambda)$-name for a $\lambda$-Aronszajn tree.
\end{enumerate}
Then forcing with
$(\MPB^2_\a \cap \mathcal{M} \ast \dot{\MPB}^1_\alpha) / \left( (\MPB^2_\a \cap \mathcal{M}) \ast (\dot{\MPB}^1_\alpha \cap \mathcal{M}_\lambda)\right)$ does not add any new branches to $\Phi(\a) \cap \mathcal{M}_\lambda$.
\end{enumerate}
\end{definition}
\begin{lemma}
\label{generalized chain condition lemma}
Suppose that $(\Phi, \Psi, \delta)$ are $(\mu, \kappa)$-suitable and $\kappa$ is weakly compact.
Then $\MPB^2_\a$ is $\kappa$-Knaster for every $\alpha \leq \delta$.
\end{lemma}

The following lemma is parallel to Lemma \ref{lemma: not adding new branches by baumgartner forcing mod subiteration}, but for the forcing $\MPB^2_\delta$, in the abstract context.

\begin{lemma}\label{lemma: not adding branches - abstract}
Assume that $\MPB^2_\delta$ is derived from a $(\mu,\kappa)$-suitable triple $(\Phi,\Psi,\delta)$ and $\kappa$ is weakly compact. Let $I \subseteq \delta$ be a set of ordinals such that $\MPB^2_I$ is a sub-iteration. Let also $S$ be a tree of height $\kappa$ in the generic extension by $\MPB^2_I$. Then, $\MPB^2_\delta / \MPB^2_I$ does not add a new cofinal branch to $S$.
\end{lemma}
\begin{proof}
Similarly to the proof of Lemma \ref{lemma: not adding new branches by baumgartner forcing mod subiteration}, we claim that the forcing \[\MPB^2_I  \ast \left(\left(\MPB^2_\delta / \MPB^2_I \right) \times \left(\MPB^2_\delta / \MPB^2_I \right)\right)\] is forcing equivalent to $\tilde{\MPB}^2_\gamma$, for some ordinal $\gamma$, where $\tilde{\MPB}^2$ is obtain by following the definition of $\MPB^2$ and modifying $\Phi$ and $\Psi$ (by inductively assuming the validity of the claim for initial segments of $I$, as in Lemma \ref{lemma: not adding new branches by baumgartner forcing mod subiteration}). Thus, it is $\kappa$-c.c. In particular, the forcing $\left(\MPB^2_\delta / \MPB^2_I \right) \times \left(\MPB^2_\delta / \MPB^2_I \right)$ is forced to be $\kappa$-c.c and thus by \cite{Unger2012},
$\MPB^2_\delta / \MPB^2_I$ does not add cofinal branches to a tree of height $\kappa$.
\end{proof}
In the next sections we will use the mechanism of this section in order to specialize trees at many cardinals simultaneously. Thus, we will need to verify that when using $\Psi$ to guess forcing notions that specialize trees, the rest of the iteration does not destroy their chain condition. 

\begin{lemma}\label{lemma: closed not adding branches}
Let $\mu < \kappa$ be regular cardinals and let $(\Phi,\Psi,\delta)$ be $(\mu,\kappa)$-suitable. Let $I\subseteq \delta$ be a set of ordinals such that $\MPB^2_I$ is a sub-iteration of $\MPB^2 = \MPB^2_\delta$. Let $S$ be a $\mu$-Aronszajn tree which is introduced by a $\mu$-c.c.\ forcing notion $\mathbb{R}$ in the generic extension by $\MPB^2_I$. Then $\MPB^2 / \MPB^2_I$ does not introduce new branches to $S$.
\end{lemma}
\begin{proof}
Note that $\MPB^2 / \MPB^2_I$ is $\mu$-closed in the generic extension by $\MPB^2_I$. Thus, by a standard argument it cannot add a cofinal branch to $S$. For the completeness of the paper, let us sketch the argument. Let $\dot{S}$ be an $\mathbb{R}$-name for an Aronszajn tree over $\MPB^2_I$. Let us assume that the quotient map $\MPB^2 / \MPB^2_I$ adds a cofinal branch, and let $\dot{b}$ be a name for this branch. 

Let us define by induction a decreasing sequence of conditions $q_i \in \MPB^2 / \MPB^2_I$ such that any condition in $\mathbb{R}$ forces that $q_i$ decides the value of the $\dot{b}$ at level $i$ (in the generic extension by $\mathbb{R}$). This is done using the chain condition of $\mathbb{R}$ and the closure of the quotient forcing $\MPB^2 / \MPB^2_I$. Thus, in the generic extension by $\mathbb{R}$ one can use the decreasing sequence $\langle q_i \mid i < \mu\rangle$ and construct a cofinal branch in $S$.
\end{proof}
We will use this lemma inductively in order to justify the preservation of the chain condition of the specialization forcings in generic extensions.
\section{The Special Aronszajn Tree Property at \texorpdfstring{$\omega$}{omega}-many successive cardinals}
\label{section3}
In this section we prove Theorem \ref{main theorem3}. The proof is based on a modification of the proof of Theorem
\ref{main theorem2}, using the abstract approach as described in Section \ref{section:abstract}, where instead of considering two successive cardinals we consider $\omega$-many of them.

Thus let $\langle \kappa_n \mid n < \omega\rangle$ be an increasing sequence of
supercompact cardinals, $\delta=(\sup_{n<\omega} \kappa_n)^{++}$
and let $\mu < \kappa_0$ be a regular cardinal\footnote{For the proof of Theorem \ref{main theorem3} it suffices to take $\mu=\aleph_0$, but here we will prove
a stronger statement that will be used in the next section for the proof of Theorem \ref{main theorem}.}.

Let us recall Laver's supercompactness indestructibility lemma, in the form that will be used in this paper.
\begin{lemma}[Laver, \cite{laver}]
\label{extended laver forcing} Let $\eta$ be a regular cardinal and $\langle \kappa_n \mid n<\omega  \rangle$ be an increasing sequence of supercompact
cardinals above $\eta$. Then there exists an $\eta$-directed closed forcing notion $\mathbb{L}(\eta, \langle \kappa_n \mid n<\omega  \rangle)$
which makes the supercompactness of each $\kappa_n$ indestructible under $\kappa_n$-directed closed forcing notions. 
\end{lemma}

By the above lemma, we may also assume that for each $n$, $\kappa_n$ is  indestructible under $\kappa_n$-directed closed forcing notions.
For notational reasons, it is convenient to denote $\kappa_{-1} = \mu$.

For each regular cardinal $\eta < \delta$, set $S^\delta_\eta = \{\a < \delta \mid \cf(\a)=\eta\}$.
Let also $\Phi: \delta \to H(\delta)$ be such that for each $x \in H(\delta)$ and $n<\omega$,~ $\Phi^{-1}(x) \cap S^\delta_{\kappa_n^+}$
is unbounded in $\delta$.

We define
an iteration
 \[
\MPB_\delta= \langle \langle \mathbb{P}_\alpha \mid \alpha \leq \delta, \rangle, \langle  \dot{\mathbb{Q}}_\a \mid \a < \delta \rangle\rangle
 \]
 of length $\delta$ as follows. During the iteration, we also define the auxiliary forcing notions $\mathbb{P}_\alpha(< \kappa_n), \mathbb{P}_\alpha(\kappa_n)$
  and $\mathbb{P}_\alpha(> \kappa_n)$, for $n<\omega, \a \leq \delta$ in such a way that
  \[
  \MPB_\a \cong \MPB_\a(> \kappa_n) \ast \dot{\MPB}_\a(\kappa_n) \ast \dot{\MPB}_\a(< \kappa_n),
  \]
where
\begin{enumerate}
\item [(a)] $\MPB_\alpha(>\kappa_n)$ is $\kappa_n$-directed closed.
\item [(b)] $\Vdash_{\MPB_\alpha(>\kappa_n)}$``$\dot{\MPB}_\alpha(\kappa_n)$ is  $\kappa_n$-c.c. and $\kappa_{n-1}$-directed closed''.
\item [(c)] $\Vdash_{\MPB_\alpha(>\kappa_n) \ast \dot{\MPB}_\a(\kappa_n)}$``$\dot{\MPB}_\alpha(<\kappa_n)$ is  $\kappa_{n-1}$-c.c. and $\mu$-directed closed''.
\end{enumerate}


 Set $\MQB_0=\prod_{n < \omega} \Col(\kappa_{n-1}, < \kappa_n)$ be the full-support product of the forcing notions $\Col(\kappa_{n-1}, < \kappa_n), n < \omega$. Let also
 \begin{enumerate}
 \item $\mathbb{P}_1( <\kappa_n) = \prod_{m < n} \Col(\kappa_{m-1}, < \kappa_m)$.
\item $\mathbb{P}_1( \kappa_n) = \Col(\kappa_{n-1}, < \kappa_n)$.
\item $\mathbb{P}_1( >\kappa_n) = \prod_{m > n} \Col(\kappa_{m-1}, < \kappa_m)$.
\end{enumerate}
Now suppose that $\alpha \leq \delta$, and that we have defined the forcing notions $\MPB_\beta$ and $\mathbb{P}_\beta(< \kappa_n), \mathbb{P}_\beta(\kappa_n), \mathbb{P}_\beta(> \kappa_n)$ for $n<\omega$ and $\beta < \alpha$. We define $\MPB_\a$,
$\mathbb{P}_\alpha(< \kappa_n), \mathbb{P}_\alpha(\kappa_n)$
  and $\mathbb{P}_\alpha(> \kappa_n)$ as follows.
A condition $p$ is in $\MPB_\a$ if and only if
\begin{enumerate}
\item $p$ has domain $\alpha$ and $\supp(p) \subseteq \bigcup_{n<\omega}S^\delta_{\kappa_n^+}$, where $\supp(p)$ denotes the support of $p$.
\item  For each $n<\omega$, $|\supp(p) \cap S^\delta_{\kappa_n^+}| < \kappa_{n-1}$.
\item If $\beta \in \supp(p) \cap S^\delta_{\kappa_n^+}$ and $\Phi(\beta)$ is a $\MPB_{\beta}( > \kappa_n) \ast \dot{\MPB}_\beta(\kappa_n) \ast \dot{\MPB}_\beta(< \kappa_n)$-name for a $\kappa_n$-Aronszajn tree, then it is forced by $\MPB_{\beta}( > \kappa_n) \ast \MPB_{\beta}(\kappa_n)$ that
     $\dot\MQB_\beta$ consists of those   partial functions $f: \Phi(\beta) \to \kappa_{n-1}$ with domain of size $<\kappa_{n-1}$, such that for every $t, s\in \dom(f)$ with $f(t) = f(s)$, we have \[1_{\MPB_\beta(<\kappa_n)}\Vdash_{\MPB_\beta(<\kappa_n)} \check t \perp_{\Phi(\beta)} \check s.\] Otherwise $\dot\MQB_\beta$
     is forced to be the trivial forcing notion.
\end{enumerate}
For $n<\omega$, $\MPB_\a(> \kappa_n)$ is defined as
\[
\MPB_\a(> \kappa_n) =\{p \in \MPB_\a \mid \supp(p) \subseteq    \bigcup_{m > n}S^\delta_{\kappa_m^+}             \}.
\]
It is
then clear that $\MPB_\a(> \kappa_n)$ is a regular subforcing of $\mathbb{P}_\alpha$.  Working in $\MPB_\a(> \kappa_n)$, the forcing notion
$\MPB_\a(\kappa_n)$ is defined as
\[
\MPB_\a(\kappa_n) =\{p \in \MPB_\a \mid \supp(p) \subseteq    S^\delta_{\kappa_n^+}, \text{compatible with } \dot{G}_{\MPB_\a(>\kappa_{n})}\}.
\]
Finally, the forcing notion $\MPB_\a(< \kappa_n)$ is defined in the generic extension by the forcing $\MPB_\a(>\kappa_n) \ast \dot\MPB_\a(\kappa_n)$ by
\[
\MPB_\a(< \kappa_n) =\{p \in \MPB_\a \mid \supp(p) \subseteq \bigcup_{m < n}S^\delta_{\kappa_m^+},\,\text{compatible with } \dot{G}_{\MPB_\a(>\kappa_{n-1})}  \}.
\]
Note that the map
\[
p \mapsto (p \upharpoonright  \bigcup_{m > n}S^\delta_{\kappa_m^+}, p \upharpoonright  S^\delta_{\kappa_n^+}, p \upharpoonright \bigcup_{m < n}S^\delta_{\kappa_m^+})
\]
defines a dense embedding from $\MPB_\a$ to $\MPB_\a(> \kappa_n) \ast \dot{\MPB}_\a(\kappa_n) \ast \dot{\MPB}_\a(< \kappa_n)$ and hence
$ \MPB_\a \cong \MPB_\a(> \kappa_n) \ast \dot{\MPB}_\a(\kappa_n) \ast \dot{\MPB}_\a(< \kappa_n)$.

Let us argue that
clauses $($a$)$-$($c$)$ continue to hold at $\alpha$. Clause $($a$)$ is evident. Clauses $($b$)$ and $($c$)$ follow from the next lemma.
\begin{lemma}
\label{chain condition for P-alpha-kappa-n}
Work in the generic extension $V[G_{\MPB_\a(> \kappa_n)}]$ by $\MPB_\a(> \kappa_n)$. Then  $\MPB_\a(\kappa_n)$
is $\kappa_{n-1}$-directed closed and $\kappa_{n}$-c.c. and $\Vdash_{\MPB_\a(\kappa_n)}$`` $\MPB_\a(<\kappa_n)$ is $\mu$-closed and $\kappa_{n-1}$-c.c.''
\end{lemma}
\begin{proof}
Work in $V[G_{\MPB_\a(> \kappa_n)}]$. It is clear that $\MPB_\a(\kappa_n)$
is $\kappa_{n-1}$-directed closed and $\Vdash_{\MPB_\a(\kappa_n)}$`` $\MPB_\a(<\kappa_n)$ is $\mu$-closed''.

As the forcing notion $\MPB_\a(> \kappa_n)$ is $\kappa_n$-directed closed, the cardinals $\kappa_m, m \leq n$,
remain supercompact in  $V[G_{\MPB_\a(> \kappa_n)}]$. Suppose also $G_{\MPB_\a(\kappa_n)}$ is $\MPB_\a(\kappa_n)$-generic over $V[G_{\MPB_\a(> \kappa_n)}]$.

Working in  $V[G_{\MPB_\a(> \kappa_n)}][G_{\MPB_\a(\kappa_n)}]$, each $\kappa_m, m < n$,
remains supercompact, and the forcing notion $\MPB_\a(<\kappa_n)$ can be seen as a finite iteration
\[
\MPB_\a(<\kappa_n) \cong \MPB_\a(\kappa_{n-1}) \ast \dots \ast \dot\MPB_\a(\kappa_0),
\]
where for each $m<n$,
\begin{enumerate}
\item $\MPB_\a(\kappa_{n-1}) \ast \dots \ast \dot\MPB_\a(\kappa_{m+1})$ is $\kappa_m$-directed closed;
\item It is forced by $\MPB_\a(\kappa_{n-1}) \ast \dots \ast \dot\MPB_\a(\kappa_{m+1})$ that the forcing notion $\MPB_\a(\kappa_m)$
specializes $\MPB_\a(\kappa_{m-1}) \ast \dots \ast \dot\MPB_\a(\kappa_{0})$-names of $\kappa_m$-Aronszajn trees.
\end{enumerate}
By (1),  $\kappa_m$ remains supercompact and hence weakly compact in the generic extension by $\MPB_\a(\kappa_{n-1}) \ast \dots  \dot\MPB_\a(\kappa_{m+1})$,
so using Lemma \ref{generalized chain condition lemma} and by induction on $m < n$,
\begin{center}
$\Vdash_{\MPB_\a(\kappa_{n-1}) \ast \dots  \dot\MPB_\a(\kappa_{m+1})}$`` $\MPB_\a(\kappa_m)$ is $\kappa_{m-1}$-directed closed and $\kappa_m$-c.c.''.
\end{center}
In particular $\Vdash_{\MPB_\a(\kappa_n)}$`` $\MPB_\a(<\kappa_n)$ is $\kappa_{n-1}$-c.c.''.

Note that in order to apply Lemma \ref{generalized chain condition lemma}, we had to make sure that whenever some name for a $\kappa_n$-Aronszajn tree is chosen in step $\gamma < \alpha$, then it is going to remain Aronszajn after forcing with $\MPB_\alpha(>\kappa_n) / \MPB_\gamma(>\kappa_n)$. This is true by the arguments of Lemma \ref{lemma: closed not adding branches}, working inductively to show that the chain condition requirements hold. 

As $\MPB_\a(> \kappa_n)$ is $\kappa_n$-directed closed, $\kappa_n$ remains supercompact and hence weakly compact in $V[G_{\MPB_\a(> \kappa_n)}]$.
  So again by Lemma \ref{generalized chain condition lemma} the forcing notion  $\MPB_\a(\kappa_n)$
is $\kappa_{n}$-c.c.
\end{proof}

 Let
 \[
 \langle \langle  G_\alpha \mid \alpha \leq \delta      \rangle, \langle  H_\alpha \mid \alpha < \delta     \rangle\rangle
 \]
 be $\MPB_\delta$-generic over $V$. Thus for each $\alpha \leq \delta, G_\alpha$ is $\MPB_\alpha$-generic over $V$,
 and if $\alpha < \delta$, then $H_\alpha$ is $\dot{\MQB}_\a[G_\a]$-generic over $V[G_\a]$.

 It is clear that
 \begin{lemma} \label{first step extension}
 $\MPB^1$ forces $\forall n<\omega, \kappa_n=\mu^{+n+1}\text{ and }2^{\kappa_n}=\kappa_{n}^+=\kappa_{n+1}$, and in particular it forces that for all $n>0$, there are special $\kappa_n$-Aronszajn trees.
\end{lemma}
\begin{proof}
We have $\MPB_1 \cong \MQB_0 = \prod_{n < \omega} \Col(\kappa_{n-1}, < \kappa_n)$, thus the first statement 
follows immediately. The second statement follows from the first one by Specker's theorem \cite{specker}.
\end{proof}

The next lemma can be proved easily using a $\Delta$-System argument.
\begin{lemma} \label{chain conditions lemma for P-delta}
For every $\alpha \leq \delta$, the forcing $\MPB_\a$ is $\delta$-c.c.
\end{lemma}
The next lemma follows from the above arguments.
\begin{lemma}
\label{cardinals by P-1 and P-delta}
The models $V[G_1]$ and $V[G_\delta]$ have the same cardinals and cofinalities. In particular,  $V[G_\delta] \models$`` for each $n<\omega$,  $\kappa_n=\mu^{+n+1}$
and $\delta = \mu^{+\omega+2}$''.
Furthermore
\[
V[G_\delta] \models \text{~``}\forall n< \omega, 2^\mu = 2^{\kappa_n} =\delta\text{~''}.
\]
\end{lemma}
By Lemmas \ref{first step extension} and \ref{cardinals by P-1 and P-delta}, we can conclude that:
\begin{lemma}
$V[G_\delta] \models$`` For each $n<\omega$,  there are $\kappa_n$-Aronszajn trees''.
\end{lemma}
The next lemma completes the proof of Theorem \ref{main theorem3}.
\begin{lemma}
\label{A-trees are specialized in the extension}
In $V[G_\delta]$, and for each $n<\omega$, all $\kappa_n$-Aronszajn trees are special.
\end{lemma}
\begin{proof}
Suppose $n<\omega$ and $T$ is a $\kappa_n$-Aronszajn tree in $V[G_\delta]$. Let $\dot{T} \in H(\delta)$
be a name for $T$. Then by our choice of $\Phi$, the set
\[
\{\alpha \in S^\delta_{\kappa_n^+} \mid \Phi(\a)=\dot{T}      \}
\]
is unbounded in $\delta$, and hence by Lemma \ref{chain conditions lemma for P-delta},
we can find some $\alpha \in S^\delta_{\kappa_n^+}$ such that $\Phi(\a)=\dot{T}$,
and $\Phi(\a)$ is a  $\MPB_{\a+1}( > \kappa_n) \ast \dot{\MPB}_\a(\kappa_n) \ast \dot{\MPB}_\a(< \kappa_n) $-name for a $\kappa_n$-Aronszajn tree.
By our definition of the forcing at step $\alpha$, we can find a function $F: T \to \kappa_{n-1}$
which is a specializing function for $T$ in $V[G_{\alpha+1}]$. As the models $V[G_\delta] \supseteq V[G_{\alpha+1}]$
have the same cardinals, $F$ witnesses that
$T$ is special in $V[G_{\delta}]$.
\end{proof}
\section{The Special Aronszajn Tree Property at successor of every regular cardinal}
\label{section4}
In this section we prove Theorem \ref{main theorem}.
Recall from Section \ref{section3}, that we essentially proved the following lemma:
\begin{lemma}
\label{extended main theorem3} Assume $\alpha$ is a limit ordinal  and $\kappa_1 < \dots < \kappa_n < \dots$ are indestructible supercompact cardinals above $\aleph_\alpha$.
Then there is an $\aleph_{\alpha+1}$-directed closed forcing notion $\MPB(\alpha, \langle \kappa_n \mid 1<n<\omega  \rangle)$ of size $\delta=(\sup_{n<\omega} \kappa_n)^{++}$ such that the following hold in a generic extension
by $\MPB(\alpha, \langle \kappa_n \mid 1< n<\omega  \rangle)$:
\begin{itemize}
\item [(a)] For each $1< n<\omega, \aleph_{\alpha+n}=\kappa_n$ and $\delta=\aleph_{\alpha+\omega+2}$.
\item [(b)] $\forall ~1 \leq n < \omega,~ 2^{\aleph_{\alpha+n}} = \delta$.
\item [(c)] The Special Aronszajn Tree Property holds at all $\aleph_{\alpha+n}$'s, $1 <n < \omega$.
\end{itemize}
\end{lemma}



Now suppose that $\langle \kappa_\xi \mid 0<\xi \in ON \rangle$ is an increasing and continuous sequence of cardinals, such that $\kappa_{\xi+1}$ is a supercompact cardinal, for every  ordinal $\xi$,
and set $\kappa_0=\aleph_0$. We also assume that no limit point of the sequence is an inaccessible cardinal.
Let
\[
 \langle   \langle \MPB_\alpha \mid \alpha \in ON, \alpha=0 \text{~or~} \lim(\alpha) \rangle, \langle  \dot{\MQB}_\alpha \mid \alpha \in ON, \alpha=0 \text{~or~}\lim(\alpha)  \rangle\rangle
\]
be the reverse Easton iteration of forcing notions such that
\begin{enumerate}
\item $\MPB_0= \{ 1_{\MPB_0} \}$ is the trivial forcing.
\item $\Vdash_{\MPB_0}$`` $\dot\MQB_0= \mathbb{L}(\aleph_1,  \langle \kappa_{n} \mid 0< n<\omega  \rangle) \ast \dot{\MPB}(0, \langle \kappa_n \mid 0<n<\omega  \rangle)$.
\item For each limit ordinal $\alpha >0$,
\begin{center}
$\Vdash_{\MPB_\alpha}$`` $\dot\MQB_\alpha = \mathbb{L}(\kappa_\alpha^+,  \langle \kappa_{\alpha+n} \mid 0< n<\omega  \rangle) \ast \dot{\MPB}(\alpha, \langle \kappa_{\alpha+n} \mid 0< n<\omega  \rangle)$''.
\end{center}

\end{enumerate}
Note that at each step $\alpha$, the forcing notion $\MPB_\alpha$ has size less than $\kappa_{\alpha+1}$, so cardinals $\kappa_{\alpha+n}, 0<n<\omega$,
remain supercompact in the generic extension by $\MPB_\alpha$. Therefore, the forcing notion $\MQB_\alpha$ is well-defined in $V[G_{\MPB_\alpha}]$.

Finally let $\MPB$ be the direct limit of the above forcing construction and let $G$ be $\MPB$-generic over $V$.
\begin{lemma}
The following hold in $V[G]$:
\begin{itemize}
\item [(a)] $\forall \xi \in ON, \aleph_\xi=\kappa_\xi$.
\item [(b)] For each limit ordinal $\alpha$  and each $1<n<\omega$,  $2^{\aleph_{\alpha+n}}= \kappa_{\alpha+\omega+2}= \aleph_{\alpha+\omega+2}$.
\end{itemize}
\end{lemma}

Let us show that in the generic extension by $\MPB$, the Special Aronszajn Tree Property holds at the successor of every regular cardinal.
Thus assume $\alpha$ is a limit ordinal (the case $\alpha=0$ is similar). We can write the forcing notion $\MPB$ as
$\MPB= \MPB_\alpha * \dot{\MQB}_\alpha * \dot{\MPB}_{(\a, \infty)}$, where, the forcing notion  $\MPB_{(\a, \infty)}$ is defined
in $V[G_{\MPB_\alpha * \dot{\MQB}_\alpha}]$, in the same way that we defined $\MPB$, using the forcing notions
$\MPB_\beta, \dot{\MQB}_\beta$, where $\alpha < \beta$ is a limit ordinal. In particular, we have
\begin{center}
$\Vdash_{\MPB_\alpha * \dot{\MQB}_\alpha}$`` $\dot{\MPB}_{(\a, \infty)}$ is $\kappa_{\alpha+\omega+1}$-closed''.
\end{center}
By Lemma \ref{extended main theorem3},
\begin{center}
$\Vdash_{\MPB_\alpha * \dot{\MQB}_\alpha}$``$\bigwedge_{1 <n<\omega} \TSP(\aleph_{\alpha+n})\text{''}$.
\end{center}
Since $\Vdash_{\MPB_\alpha * \dot{\MQB}_\alpha}$``the forcing notion $\MPB_{(\a, \infty)}$ does not add any new $\kappa_{\alpha+\omega}$-sequences'',
we have
\begin{center}
 $\Vdash_{\MPB}$``$\bigwedge_{1 <n<\omega} \TSP(\aleph_{\alpha+n})$''.
\end{center}
The result follows immediately. \hfill$\Box$

We close the paper with the following question, which is an analogue of Magidor's question regarding the Tree Property.
\begin{question}
Is it consistent, relative to the existence of large cardinals, that Special Aronszajn Tree Property holds for all uncountable regular cardinals?
\end{question}
Let us also remark that the following question is still open:
\begin{question}
Let $\lambda$ be  successor of a singular cardinal. Is $\TSP(\lambda)$ consistent? I.e.,\ is it consistent that there is a $\lambda$-Aronszajn tree, and every $\lambda$-Aronszajn tree is special?
\end{question}
\section{Acknowledgments}
We would like to thank John Krueger for pointing us to an error in a previous version of the paper. 

We would like to thank the anonymous referee for many useful remarks and in particular for pointing us to some problems in a previous version of the paper and helping us to improve the readability of this paper. We would like to thank the referee for pointing an error in the argument for Claim \ref{separation claim} and suggesting a way to fix the proof.

\providecommand{\bysame}{\leavevmode\hbox to3em{\hrulefill}\thinspace}
\providecommand{\MR}{\relax\ifhmode\unskip\space\fi MR }
\providecommand{\MRhref}[2]{%
  \href{http://www.ams.org/mathscinet-getitem?mr=#1}{#2}
}
\providecommand{\href}[2]{#2}


\begin{thebibliography}{10}

\bibitem{Abraham1983}
Uri Abraham, \emph{Aronszajn trees on {$\aleph _{2}$} and {$\aleph _{3}$}},
  Ann. Pure Appl. Logic \textbf{24} (1983), no.~3, 213--230. \MR{717829}

\bibitem{baumgartner}
James Baumgartner, Jerome Malitz, and William Reinhardt, \emph{Embedding trees
  in the rationals}, Proceedings of the National Academy of Sciences
  \textbf{67} (1970), no.~4, 1748--1753.

\bibitem{baumgartner1983}
James~E. Baumgartner, \emph{Iterated forcing}, Surveys in set theory, London
  Math. Soc. Lecture Note Ser., vol.~87, Cambridge Univ. Press, Cambridge,
  1983, pp.~1--59. \MR{823775}

\bibitem{hauser}
Kai Hauser, \emph{Indescribable cardinals and elementary embeddings}, J.
  Symbolic Logic \textbf{56} (1991), no.~2, 439--457. \MR{1133077}

\bibitem{jech2003}
Thomas Jech, \emph{Set theory}, Springer Monographs in Mathematics,
  Springer-Verlag, Berlin, 2003, The third millennium edition, revised and
  expanded. \MR{1940513}

\bibitem{laver}
Richard Laver, \emph{Making the supercompactness of $\kappa$ indestructible
  under $\kappa$-directed closed forcing}, Israel Journal of Mathematics
  \textbf{29} (1978), no.~4, 385--388.

\bibitem{laver-shelah}
Richard Laver and Saharon Shelah, \emph{The $\aleph_2$-souslin hypothesis},
  Transections of the American Mathematical Society \textbf{264} (1981), no.~2,
  411--417.

\bibitem{MagidorShelah1996}
Menachem Magidor and Saharon Shelah, \emph{The tree property at successors of
  singular cardinals}, Arch. Math. Logic \textbf{35} (1996), no.~5-6, 385--404.
  \MR{1420265}

\bibitem{Mitchell1972}
William Mitchell, \emph{Aronszajn trees and the independence of the transfer
  property}, Ann. Math. Logic \textbf{5} (1972/73), 21--46. \MR{0313057}

\bibitem{Neeman2014}
Itay Neeman, \emph{The tree property up to {$\aleph_{\omega+1}$}}, J. Symb.
  Log. \textbf{79} (2014), no.~2, 429--459. \MR{3224975}

\bibitem{Rinot2017}
Assaf Rinot, \emph{Higher {S}ouslin trees and the {GCH}, revisited}, Adv. Math.
  \textbf{311} (2017), 510--531. \MR{3628222}

\bibitem{ShelahThomas1997}
Saharon Shelah and Simon Thomas, \emph{The cofinality spectrum of the infinite
  symmetric group}, J. Symbolic Logic \textbf{62} (1997), no.~3, 902--916.
  \MR{1472129}

\bibitem{SolovayTennenbaum}
R.~M. Solovay and S.~Tennenbaum, \emph{Iterated {C}ohen extensions and
  {S}ouslin's problem}, Ann. of Math. (2) \textbf{94} (1971), 201--245.
  \MR{0294139}

\bibitem{specker}
E.~Specker, \emph{Sur un probl\`eme de {S}ikorski}, Colloquium Math. \textbf{2}
  (1949), 9--12. \MR{0039779}

\bibitem{Unger2012}
Spencer Unger, \emph{Fragility and indestructibility of the tree property},
  Arch. Math. Logic \textbf{51} (2012), no.~5-6, 635--645. \MR{2945572}

\bibitem{Unger2015}
\bysame, \emph{Fragility and indestructibility {II}}, Ann. Pure Appl. Logic
  \textbf{166} (2015), no.~11, 1110--1122. \MR{3385103}

\end{thebibliography}
\end{document}